\documentclass{article}
\RequirePackage[OT1]{fontenc}
\usepackage[dvipsnames]{xcolor}
\RequirePackage{amsthm,amsmath,amsbsy,amssymb,array,bm,graphicx,epstopdf,mathtools,calc,xfrac,enumitem,microtype}
\RequirePackage[numbers,sort&compress]{natbib}
\RequirePackage{hyperref}
\hypersetup{colorlinks=true,citecolor=MidnightBlue,urlcolor=MidnightBlue,linkcolor=PineGreen,breaklinks} 
\usepackage[ruled,linesnumbered]{algorithm2e}
\SetKwInput{KwInput}{Input}
\SetKwInput{KwOutput}{Output}
\usepackage{subfigure,subdepth,mdframed,tikz,pgf,pgffor}
\usepgfmodule{shapes,plot}
\usetikzlibrary{decorations,arrows,positioning,calc}
\xdefinecolor{darkgreen}{RGB}{0,193,0}
\def\dgrn{darkgreen}
\newlength{\mylength}
\allowdisplaybreaks

\tikzstyle{dir}= [postaction={decorate,
	decoration={markings,
	mark=at position .65
	with {\arrow[scale=1]{stealth}}}}]
\tikzstyle{dirs}= [postaction={decorate,
	decoration={markings,
	mark=at position .65
	with {\arrow[scale=.8]{stealth}}}}]
\tikzstyle{nd} = [circle, fill=black,
	inner sep=0pt,
	minimum width=3 pt]
\tikzstyle{bnd} = [circle,fill=black,
	inner sep=1pt,
	minimum width=6 pt,
	text=white]
\tikzstyle{rnd} = [circle, fill=black!30,
	inner sep=0pt,
	minimum width=3 pt]
\tikzstyle{brnd} = [circle,fill=black!30,
	inner sep=1pt,
	minimum width=6 pt,
	text=black]
\newcommand{\TexttriangleR}{%
\raisebox{-0.45mm}{\!\!
\tikz[scale=.2,clip]{\draw[thick]
(210:.6) node[rnd] {}--
(  90:.6) node[nd]{}--
( -30:.6) node[nd]{}--
(210:.6) node[rnd]{};
}}\hspace{-.9mm}
}
\newcommand{\TextsquareR}{%
\raisebox{-0.45mm}{\!
\tikz[scale=.2,clip]{\draw[thick]
(        45:.65) node[nd]{}--
(  45+90:.65) node[nd]{}--
(45+180:.65) node[rnd]{}--
(45+270:.65) node[nd]{}--
(        45:.65) node[nd]{};
}}\hspace{-.9mm}
}
\newcommand{\TextshovelR}{%
\raisebox{-0.45mm}{\!
\tikz[scale=.2,clip]{\draw[thick]
(        45:.65) node[nd]{}--
(  45+90:.65) node[nd]{}--
(45+180:.65) node[rnd]{}--
(45+270:.65) node[nd]{}--
(  45+90:.65) node[nd]{};
}}\hspace{-.9mm}
}
\newcommand{\TexttrypodR}{%
\raisebox{-0.45mm}{\!
\tikz[scale=.2,clip]{\draw[thick]
(        45:.65) node[nd]{}--
(  45+90:.65) node[nd]{}--
(45+180:.65) node[rnd]{}--
(  45+90:.65) node[nd]{}--
(       -45:.65) node[nd]{};
}}\hspace{-.9mm}
}
\newcommand{\TextdiamondR}{%
\raisebox{-0.45mm}{\!
\tikz[scale=.2,clip]{\draw[thick]
(        45:.65) node[nd]{}--
(  45+90:.65) node[nd]{}--
(45+180:.65) node[rnd]{}--
(45+270:.65) node[nd]{}--
(        45:.65) node[nd]{}--
(45+180:.65) node[rnd]{};
}}\hspace{-.5mm}
}
\newcommand{\TextbowtieR}{%
\raisebox{-0.45mm}{\!
\tikz[scale=.2,clip]{\draw[thick]
(        35:.8) node[nd]{}--
(          0:   0) node[rnd]{}--
(35+180:.8) node[nd]{}--
(-35-180:.8) node[nd]{}--
(          0:   0) node[rnd]{}--
(       -35:.8) node[nd]{}--
(        35:.8) node[nd] {};
}}
}
\newcommand{\TextDoubleTwoPathsR}{%
\raisebox{-0.45mm}{\!
\tikz[scale=.2,clip]{\draw[thick]
(        35:.8) node[nd]{}--
(          0:   0) node[rnd]{}--
(35+180:.8) node[nd]{}--
(-35-180:.8) node[nd]{} 
(          0:   0) node[rnd]{}
(       -35:.8) node[nd]{}--
(        35:.8) node[nd] {};
}}
}
\newcommand{\TextTriangleTwoPathsR}{%
\raisebox{-0.45mm}{\!
\tikz[scale=.2,clip]{\draw[thick]
(        35:.8) node[nd]{}--
(          0:   0) node[rnd]{}--
(35+180:.8) node[nd]{}--
(-35-180:.8) node[nd]{}--
(          0:   0) node[rnd]{}
(       -35:.8) node[nd]{}--
(        35:.8) node[nd] {};
}}
}
\newcommand{\TextcherryR}{%
\raisebox{-0.45mm}{\!\!
\tikz[scale=.2,clip]{\draw[thick]
(210:.6) node[rnd] {}--
(  90:.6) node[nd]{}--
( -30:.6) node[nd]{};
}}\hspace{-.9mm}
}
\allowdisplaybreaks
\makeatletter
\newcommand{\pushright}[1]{\ifmeasuring@#1\else\omit\hfill$\displaystyle#1$\fi\ignorespaces}
\newcommand{\pushleft}[1]{\ifmeasuring@#1\else\omit$\displaystyle#1$\hfill\fi\ignorespaces}
\makeatother
\let\originalleft\left%
\let\originalright\right%
\renewcommand{\left}{\mathopen{}\mathclose\bgroup\originalleft}%
\renewcommand{\right}{\aftergroup\egroup\originalright}%
\makeatletter%
\def\mybigx#1{\dimen@#1\relax%
\mathchoice%
{\vbox to \dimen@{}}%
{\vbox to \dimen@{}}%
{\vbox to .7\dimen@{}}%
{\vbox to .5\dimen@{}}}%
\def\mybig#1{{\hbox{$\left#1\mybigx{0.8em}\right.\n@space$}}}%
\makeatother%
\renewcommand{\P}{\ensuremath{\operatorname{\mathbb{P}}}}
\newcommand{\E}{\ensuremath{\operatorname{\mathbb{E}}}}
\newcommand{\II}{\ensuremath{\operatorname{\mathbb{I}}}}
\newcommand{\var}{\ensuremath{\operatorname{Var}}}
\newcommand{\cov}{\ensuremath{\operatorname{Cov}}}
\newcommand{\aut}{\ensuremath{\operatorname{aut}}}

\DeclareFontFamily{U}{mathx}{\hyphenchar\font45}
\DeclareFontShape{U}{mathx}{m}{n}{<-> mathx10}{}
\DeclareSymbolFont{mathx}{U}{mathx}{m}{n}
\usepackage{thmtools}
\usepackage{thm-restate}

\usepackage[noabbrev,capitalise,nameinlink]{cleveref}
\numberwithin{equation}{section}
\theoremstyle{plain}
\crefname{equation}{}{}
\crefname{figure}{Fig.}{Figs.}
\crefname{table}{Table}{Tables}
\newtheorem{Definition}{Definition}
\crefname{Definition}{Definition}{Definition}
\crefname{defin}{Definition}{Definition}
\newtheorem{Lemma}{Lemma}
\crefname{Lemma}{Lemma}{Lemmas}
\newtheorem{Proposition}{Proposition}
\crefname{Proposition}{Proposition}{Propositions}
\newtheorem{Theorem}{Theorem}
\crefname{Theorem}{Theorem}{Theorems}
\crefname{thm}{Theorem}{Theorems}

\crefname{Corollary}{Corollary}{Corollaries}
\crefname{coro}{Corollary}{Corollaries}
\newtheorem{Remark}{Remark}
\crefname{Remark}{Remark}{Remarks}

\title{Central limit theorems for local network statistics}%
\author{P-A. Maugis, University College London}
\date{}
\begin{document}%%%%%%%%%%%%%%%%%%%%%%%%
%%%%%%%%%%%%%%%%%%%%%%%%%%%%%%%%%
\lineskiplimit=-2pt
\linespread{1.1}
\abovedisplayskip= .7\baselineskip minus .7\baselineskip
%%%%%%%%%%%%%%%%%%%%%%%%%%%%%%%%%
\maketitle
\begin{abstract}
Subgraph counts---in particular the number of occurrences of small shapes such as triangles---char\-ac\-terize properties of random networks, and as a result have seen wide use as network summary statistics. However, subgraphs are typically counted globally, and existing approaches fail to describe vertex-specific characteristics. On the other hand, rooted subgraph counts---counts focusing on any given vertex's neigh\-bor\-hood---are fundamental descriptors of local network properties. We derive the asymptotic joint distribution of rooted subgraph counts in inhomogeneous random graphs, a model which generalizes many popular statistical network models. This result enables a shift in the statistical analysis of large graphs, from estimating network summaries, to estimating models linking local network structure and vertex-specific covariates. As an example, we consider a school friendship network and show that local friendship patterns are significant predictors of gender and race.
\end{abstract}

%%%%%%%%%%%%%%%%%%%%%%%%%%%
\section{Introduction}%%%%%%%%%%%%%%%%%
%%%%%%%%%%%%%%%%%%%%%%%%%%%
We contribute the results---and associated statistical inference methods---to enable the study of large networks at the level of each vertex. To do so we first derive the limit joint distribution of vertex-specific subgraph counts in inhomogeneous random graphs. Then, we exhibit the usefulness of these counts by studying a network of school friendships and showing that local counts are significant predictors of gender and race.

A graph, or network, is composed of vertices connected by edges. For instance, students connected by friendship ties in a school. Such a graph may be studied at two different scales: one global, the other local. In our school friendship example, the former would consider friendship patterns throughout the school. The latter would study potential links between local patterns surrounding a student and other known features (such as gender or race). The former is well studied, and powerful tools exist to study graphs at the global scale~\cite{BickelLevina2012,bhattacharyya2013subsampling,rohe2011,tang2013,gao2015,klopp2016,rinaldo2013,sarkar2015,lei2015,zhao2012consistency,ambroise2011new}. We focus on the latter.

To study graphs at the local scale we introduce rooted subgraph counts. Such a count is calculated from a graph by counting the occurrences of a particular shape attached to a given vertex. The simplest such count is a vertex's degree: the number of edges attached to a vertex (\cref{EgoMotif}.(a) and (d)). The second most common rooted count is the number of triangles attached to a vertex (\cref{EgoMotif}.(c) and (f)). Rooted counts provide vertex-specific information: for instance, in our school friendship example, depending on how many edges and triangles a student is connected to, this student will play a different role in a contagion process~\cite{isham2011spread}. Effectively, rooted counts recover powerful fundamental characterizers of local graph properties~\cite[restricted homomorphism]{lovasz2014automorphism}, and are key features used to compare and classify biological networks~\cite[graphlet degree]{przulj2007,ali2014alignment,ahmed2017}. 

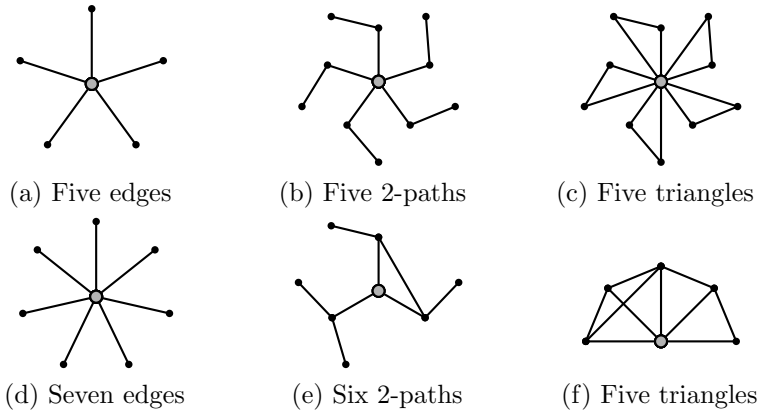
\begin{figure}[!t]
	\centering
 	\def\s{1}
		\tikzstyle{pop}=[circle, draw, fill=black!30,
				inner sep=0pt, minimum width=4.5pt]
	\tikzstyle{pip}=[circle, draw, fill=black,
				inner sep=0pt, minimum width=2pt]
		\tikzstyle{bluepop}=[circle, draw, color=blue,fill=blue!75,
						inner sep=0pt, minimum width=4.5pt]
		\tikzstyle{popi}=[circle, draw, color=blue!75,fill=blue!75,
						inner sep=0pt, minimum width=2pt]
		\tikzset{pil/.style={thick,color=black!30}}
		\tikzset{pul/.style={thick,color=blue!50}}
	{%\footnotesize
	
	%%%%%%%%%%%%%%%%%%%%%%%%%%%%%
	\noindent
	\begin{tabular}{ccc}
	\begin{minipage}[b][1.8cm]{.2\mylength}\centering\vfill
	\def\n{4}
	\def\p{5}
	\begin{tikzpicture}[thick,scale=\s*.5]
		\foreach \j in {0,...,\n}
			\draw {(0:0) node[pop] {} -- (90+\j*360/\p:2) node[pip] {}};
%		\foreach \j in {0,...,\n}
%			\draw {(0:0) node[popi] {} edge[pul] (90+\j*360/\p:2) node[bluepop] {}};
%		\foreach \j in {0,...,\n}
%			\draw {(90+\j*360/\p:2) node[popi] {}};
	\end{tikzpicture}
    \vspace{0.5\baselineskip}%
	\vfill
	{(a) Five edges}
    \vspace{1\baselineskip}%
	\end{minipage}
	%%%%%%%%%%%
	\begin{minipage}[b][1.8cm]{.2\mylength}\centering\vfill
	\def\n{4}
	\def\p{5}
	\begin{tikzpicture}[thick,scale=\s*.357]
		\foreach \j in {0,...,\n}
			\draw {(0:0) node[pop] {}
					-- (90+\j*360/\p:2) node[pip] {}
					-- (90+36+\j*360/\p:3) node[pip] {}};
	\end{tikzpicture}
	\vfill
	{(b) Five 2-paths}
    \vspace{\baselineskip}%
	\end{minipage}
	%%%%%%%%%%%
	\begin{minipage}[b][1.8cm]{.2\mylength}\centering\vfill
	\def\n{4}
	\def\p{5}
	\begin{tikzpicture}[thick,scale=\s*.357]
		\foreach \j in {0,...,\n}
			\draw {(0:0) node[pop] {}
					-- (90+\j*360/\p:2) node[pip] {}
					-- (90+36+\j*360/\p:3) node[pip] {}
					-- (0:0) node {}};
	\end{tikzpicture}
	\vfill
	{(c) Five triangles}
    \vspace{\baselineskip}%
	\end{minipage}
	\tabularnewline
	%\\[-\baselineskip]
	%%%%%%%%%%%%%%%%%%%%%%%%%%%%%
	\begin{minipage}[b][2cm]{.2\mylength}\centering\vfill
	\def\n{6}
	\def\p{7}
	\begin{tikzpicture}[thick,scale=\s*.5]
		\foreach \j in {0,...,\n}
			\draw {(0:0) node[pop] {}  -- (90+\j*360/\p:2) node[pip] {}};
	\end{tikzpicture}
	\vfill
	{(d) Seven edges}
	\end{minipage}
	%%%%%%%%%%%
	\begin{minipage}[b][2cm]{.2\mylength}\centering\vfill
	\def\n{2}
	\def\p{3}	
	\begin{tikzpicture}[thick,scale=\s*.357]
		\foreach \j in {0,...,\n}
			\draw {(0:0) node[pop] {}
					-- (90+\j*360/\p:2) node[pip] {}
					-- (90+36+\j*360/\p:3) node[pip] {}};
		\draw {(210:2) node[pip] {} -- (210-36:3) node[pip] {}};
		\draw {(90:2) node[pip] {} -- (330:2) node[pip] {}};
	\end{tikzpicture}
	\vfill
	{(e) Six 2-paths}
	\end{minipage}
	%%%%%%%%%%%
	\begin{minipage}[b][2cm]{.2\mylength}\centering\vfill
	\def\n{3}
	\def\p{4}
	\begin{tikzpicture}[thick,scale=\s*.5]
		\foreach \j in {0,...,\n}
			\draw {(0:0) node[pop] {}
					-- (\j*180/\p:2) node[pip] {}
					-- (\j*180/\p+180/\p:2) node[pip] {}};
		\draw{(0:0) node[pop] {} -- (\n*180/\p+180/\p:2) node[pip] {}};
		\draw{(3*180/\p+180/\p:2) node[pip] {} -- (180/\p+180/\p:2) node[pip] {}};
	\end{tikzpicture}
	\vfill\vfill
	{(f) Five triangles}
	\end{minipage}
	%%%%%%%%%%%%%%%%%%%%%%%%%%%%
	\end{tabular}}

	\caption{Diagrams illustrating rooted subgraph counts. In each subfigure we consider the rooted subgraph count of the central vertex, highlighted in grey. The rooted subgraphs counted are, in the first to third column, an edge, a path of length 2, and a triangle.}
	\label{EgoMotif}
\end{figure}

To undertake vertex-specific inference we must quantify the variations of rooted counts expected under a null. Specifically, a description of the ranges of local behaviors that can be explained by known global network features, such as sparsity, heavy tailed degree distributions and community structure. The null of inhomogeneous random graphs~\citep{bollobas2007phase}, which subsumes most models used in the statistical literature~\cite{BickelLevina2012,bhattacharyya2013subsampling,rohe2011,tang2013,gao2015,klopp2016,rinaldo2013,airoldi2008mixed,sarkar2015,lei2015,zhao2012consistency,hoff2002latent,hoff2007modeling,young2007random,Sussman2012}, is recognized as sufficient to model these three global features. The question we address then is: ``Are observed local behaviors consistent with this model?''

Surprisingly, we observe that while inhomogeneous random graphs do not explicitly model local behavior---something transparent in our results below---the null is sufficient to account for local behavior in several social networks. Specifically, while the null can be rejected as a model for local behavior across the whole network at once, it cannot be rejected as a model to describe local behavior within covariate-informed subnetworks; in our school friendship example, we find that the local behavior within grades is accounted for, but not between grades. Further studies, that our tools enable, could look into determining the intermediate scale at which the null starts to fail, and characterizing how local behavior departs from the null.

Underpinning this finding are two main contributions. The first is to derive the asymptotic joint distribution of rooted counts under the null. Surprisingly, no standard results or proof techniques exist to study rooted counts---proof techniques for global counts rely on the overlap between copies being limited~\cite{nowicki1988,barbour1989,BickelLevina2012,coulson2016poisson}, while this does not hold for rooted copies---and very little is known about rooted counts even in the Erd\H{o}s-R\'enyi model~\cite{rucinski1986balanced,spencer1990extensions,janson2011}. To prove our result we rely on the intuition of the historically oldest moment-based methods for global counts~\cite{erdos59:_on_random_graphs,bollobas1981threshold,rucinski1988small}: that moments of subgraph counts are tied together by an algebraic structure. To uncover such a structure between rooted counts, we follow the proof progression of~\cite{lovasz2012large} and first present a quasi homomorphism tying together products of rooted counts. This then enables us to determine how rooted copies are most likely to overlap, elicit the sought-after algebraic structure, and eventually deduce all the moments of rooted counts. We find that odd central moments are negligible, while even central moments follow a double factorial progression, which ultimately allows us to resolve the rooted counts' limiting distribution as Normal. 

Our second main result is to prove that the sample of rooted counts---the array of rooted counts across vertices---can be treated, for the purpose of statistical inference, as an independent and identically distributed (i.i.d.) sample; e.g., one can perform maximum likelihood estimation, obtain the Fisher information matrix and produce confidence intervals in the standard way (see \cref{ml-estimation}). This enables fully non-parametric estimation of models linking rooted counts and vertex-specific covariates using generic statistical tools. This is key for applications, as existing methods are hard to implement, do not enable the use of covariates, or scale so well to large graphs. A large literature exists on efficient counting of rooted graphs; e.g.,~\cite{przulj2007,ali2014alignment,ahmed2017}.

Both results are the best of their kind: we obtain convergence in moments for any rooted count, but find that uniform convergence does not hold; we obtain uniform convergence for averages of rooted counts, but find that convergence to be arbitrarily slow. In both cases it is the strength of the dependence that limits the convergence. This leads to additional care being needed when working on data (e.g.,~\cref{cval}.)

\paragraph{Outline} \cref{Elicitation} introduces the null model and the definitions needed to treat rooted counts. \cref{Theory} presents two new central limit results for rooted counts. \cref{Method} presents: a goodness-of-fit test that identifies the vertices causing the lack of fit; and a regression, with associated confidence regions, linking rooted counts and vertex-specific covariates. \cref{Discussion} concludes. Proofs can be found in \cref{proof-motif-density-limit} and \ref{ProofCLT}. Simulation experiments and algorithms can be found in \cref{metho}.

%%%%%%%%%%%%%%%%%%%%%%%%%%%%%%%%%%%%%
\section{Elicitation: subgraphs in random graphs}%%%%
\label{Elicitation}%%%%%%%%%%%%%%%%%%%%%%%%%%%%
%%%%%%%%%%%%%%%%%%%%%%%%%%%%%%%%%%%%%
We now introduce our null model and our statistic. Our null model is that of generalized, or kernel-based, random graphs~\cite{bollobas2007phase}. Our statistics are rooted subgraph densities, and are related to the statistics used in~\cite{rucinski1986balanced,spencer1990extensions,przulj2007,ali2014alignment}.

As in~\cite{BickelLevina2012,zhao2012consistency,gao2015}, to allow for sparsity, community structure and power-law degree distribution, we use inhomogeneous random graphs~\cite{bollobas2007phase}:

\begin{Definition}[Inhomogeneous random graphs sequence $G(\rho,\kappa)$]
\phantomsection\label{model}%%%
Let $\kappa\in L^\infty([0,1]^2)$ be a {\emph kernel} (i.e., symmetric, positive, bounded by $1$, and such that $\int_{{}^{_{[0,1]}}}\!\kappa(x,\cdot)dx>0$) and $\rho = (\rho_n)_{n>0}$ be a sequence in $(0,1)$. We call $G(\rho,\kappa) = (G_n)_{n>0}$ the sequence of random graphs such that for each $n$, $G_n$ is the random graph on $[n]$ where to each vertex $i$ is associated independently $x_i\sim\operatorname{Uniform}([0,1])$ and such that independently across pairs $\{i,j\}\subset[n]$,
\[\II\{ij\in G_n\}\,\vert\, x_i,x_j \sim \operatorname{Bernoulli}\!\big(\rho_n\kappa(x_i,x_j)\big).\]
\end{Definition}%%%

\cref{model} subsumes, among others, blockmodels, latent space models~\cite{hoff2002latent,hoff2007modeling} and the random dot-product models~\cite{nickel2006random,young2007random,Sussman2012,athreya2013limit}. Such models assign to each vertex a position in a space, say $\bm{z}_i\in \mathcal{X}\subset\smash{\mathbb{R}^k}$, and the probability of an edge is tied to a function $\phi$ of the $\bm{z}_i$-s; i.e., $\mathbb{P}\big(ij\in G\,\vert\, \bm{z}_i,\bm{z}_j\big) = \phi(\bm{z}_i,\bm{z}_j)$. This is equivalent to assuming that $\kappa$ may be written as $\kappa(x_i,x_j) = \phi(\psi(x_i),\psi(x_j))$, with $\psi:[0,1]\to\mathcal{X}$. Note that by Borel isomorphism theorem, $\psi$ can be bijective, and that setting the $x_i$-s in $[0,1]$ in \cref{model} is not restrictive.

In the following we focus on asymptotic results in the limit of large $n$. Especially, we shall consider the sparse regimes where $\rho_n\to0$. To address this problem, we define our statistic as the normalize rooted count, thereby yielding statistics that remains order one in the limit:

\begin{Definition}[Rooted graph~\protect{\cite{harary1955}}]
 \phantomsection\label{defn:rooted-graph}%%%
A rooted graph $F = [F,v]$ is a labeled graph $F$ in which a vertex $v$ has been singled out. Two rooted graphs $F=[F,v]$ and $F'=[F',v']$ are isomorphic, and we write $F\equiv F'$, if there exists an adjacency preserving isomorphism that maps $F$ onto $F'$ and $v$ onto $v'$.
\end{Definition}

\begin{Definition}[Rooted subgraph density]
\phantomsection\label{defn:rooted-count}%%
Fix a rooted graph $F$, a simple graph $G$ of order $n$, and a vertex $i\in G$. Denoting $K_n$ the complete graph on $n$ vertices and $e(H)$ the number of edges in a graph $H$, we write
\begin{equation*}
s_i(F,G) = \left(\frac{e(G)}{e(K_n)}\right)^{\!\!-e(F)}\frac{\#\{F'\subset G_{\phantom{n}}:[F',i]\equiv F\}}{\#\{F'\subset K_n:[F',i]\equiv F\}}
\end{equation*}
for the rooted density of $F$ at vertex $i\in G$.
\end{Definition}%%%

\Cref{defn:rooted-count} is standard---it may be understood as the normalized probability of $|F|-1$ uniformly selected vertices forming a copy of $F$ rooted at $i$---and relates directly to others; e.g.,
~\citet[rooted graph]{janson2011};~\citet[graph extension]{spencer1990extensions};~\citet[restricted homomorphism]{lovasz2014automorphism};~\citet[graphlet]{ali2014alignment,przulj2007};~\citet[graph in vertex's ego-network]{moustafa:icde12,cunningham2013ego};~\citet[role sequence]{Miller2009percolation,Newman2009clustering,Karrer2010subgraphs}. The global subgraph density, which is more common in the literature, is up to a constant factor, the sum across vertices of the rooted densities~\cite{Lovasz2006limits,bollobas2007phase,BickelLevina2012,bhattacharyya2013subsampling}.

With this definition, we will show in \cref{Theory} that as long as the average degree grows with $n$---i.e., if $n\rho_n\to\infty$, which is the minimal assumption for statistical inference under our null model~\cite{BickelLevina2012,zhao2012consistency,gao2015}---then for $F$ a tree, the $s_i(F,G_n)$ are asymptotically Normal. Therefore, surprisingly, there is no gap between the minimal assumption necessary for local and global inferences in inhomogeneous random graphs.

%%%%%%%%%%%%%%%%%%%%%%%%%%%%%%%%%%%%%
\section{Theory: limit distribution of rooted densities}%%%%%%
\label{Theory}%%%%%%%%%%%%%%%%%%%%%%%%%%%%%%
%%%%%%%%%%%%%%%%%%%%%%%%%%%%%%%%%%%%%

We now present our two main results. The first characterizes the limit distribution of rooted densities at any vertex. The second is a central limit result for the average, taken across vertices, of functions of rooted densities. In \cref{Method} we will show through an example that these results lead to new methods for statistical inference on network data. 

\subsection{Limit of rooted densities} Our proofs rely on the method of moments: we prove moment convergence to obtain convergence in distribution. \citet{erdos59:_on_random_graphs}, \citet{bollobas1981threshold}, and ultimately~\citet{rucinski1988small}, exhibited the power of this proof technique to study global subgraph counts. Especially, we find that although $U$-statistics and Poisson convergence based methods are easier to implement, only the method of moments handles the strong dependence that may exist between rooted copies, and yields sufficient and necessary conditions for any subgraph $F$ and density regime $\rho$.

The power of the moment method, when applied to subgraph counts, rests crucially upon the very simple asymptotic behavior of the expected number of rooted copies. Specifically, if $(G_n)\sim G(\rho,\kappa)$, then for any rooted graph $F$ and vertex $i$ we have (see \cref{first-order})
\[
\E \#\{F'\subset G_n\,:\, [F',i]\equiv F\} = \Theta\big(n^{|F|-1}\rho_n^{e(F)}\big).
\]
Thus, the simplest features of $F$---its number of vertices and edges---char\-ac\-te\-rize the first order behavior of its rooted count; e.g., by Markov's inequality, there are asymptotically no copies of $F$ rooted at $i$ if $n\rho_n^{e(F)/(|F|-1)}\to0$.

Using a proof technique inspired by~\citet{bollobas1981threshold},~\citet{rucinski1988small} and~\citet{lovasz2012large}, we build upon the simple structure of the first moment to char\-ac\-te\-rize higher order moments. To do so, we first introduce $\mathcal{H}_F$, the set of rooted graphs that can be built from two copies of $F$ sharing their root and at least one other vertex: for the rooted triangle $\TexttriangleR\ $(root vertex in grey), $\mathcal{H}_{\TexttriangleR} = \{\! \TextdiamondR\ , \TexttriangleR\,\}$; for the rooted cherry $\TextcherryR\ $, $\mathcal{H}_{\TextcherryR} = \{\!\TextsquareR\ ,\TextshovelR\ ,\TexttrypodR\ ,\TexttriangleR\ , \TextcherryR\,\}$. Then, we find that the variance of the number of rooted copies is of the same order of magnitude as the sum of the number of copies of elements of $\mathcal{H}_F$.

We then show that the $2k$-th moment of the number of rooted copies is also driven by $\mathcal{H}_F$. This higher order moment is equal, up to a negligible term, to the total number of copies of all the subgraphs that can be built from $k$ elements of $\mathcal{H}_F$ sharing only their root. This enables us to show that $2k$-th moment is asymptotically equal to $(k-1)!!$ times the variance to the $k$-th power, and then---invoking the Hausdorff moment method---that the $s_i(F,G_n)$ are asymptotically Normal (proof to be found in \cref{proof-motif-density-limit}):

\begin{restatable}[Limit of rooted densities]{thm}{thmlocal}
\phantomsection\label{motif-density-limit}%%%%%
\vspace{-.5\baselineskip}
Set $(G_n) \sim G(\rho,\kappa)\,|\,x_i = x$ and define the parameter $m = \max\{e(H)/(|H|-1):v\in H\subset F, |H|>1\}$. Then, asymptotically in $n$,
\begin{equation*}
s_i(F,G_n)\to_p 
\begin{cases}
0 & \text{if } n\rho_n^m\to 0,\\
s_x(F,\kappa) : = 
\E\left[\,\prod_{pq\in F}\kappa(x_p,x_q)\Big | x_i=x\,\right] & \text{if } n\rho_n^m\to\infty.
\end{cases}
\end{equation*}
Furthermore, in the later case, and if $s_x(F,\kappa)>0$,
\begin{equation*}
\frac{s_i(F,G_n)-s_x(F,\kappa)}{\sqrt{\var s_i(F,G_n)}}
\xrightarrow[\quad ]{\mathcal{L}}
\mathrm{Normal}(0,1).
\end{equation*}
The result extends to the multivariate setting: a vector of standardized rooted subgraph densities is asymptotically multivariate Normal.
\vspace{-.5\baselineskip}
\end{restatable}

\cref{motif-density-limit} generalizes~\citet[Theorem~5]{rucinski1986balanced} which considers sums of rooted counts of balanced subgraphs, and~\citet[Theorem~5]{spencer1990extensions} which considers the rate at which rooted counts concentrate, both under a constant kernel. A formula for $\var s_i(F,G_n)$ is provided in \cref{var-closed-form}.

The rate at which $\var s_i(F,G_n)$ shrinks, and therefore the rate of convergence toward the Normal distribution, changes with $\rho$. This parallels the behavior of global subgraph counts in the Erd\H{o}s-R\'enyi model~\cite{rucinski1988small}. If $n\smash{\rho_n^m}$ diverges slowly, copies of $F$ are likely to substantially overlap, and the rate depends in complex way on the structure of $F$. However, if $n\smash{\rho_n^m}$ diverges fast enough, copies tend to only overlap over one edge, and the rate is $\sqrt{n\rho_n}$. Finally, if $n\smash{\rho_n^m}=\Theta(1)$, although all moments converge, in general no limit distribution can be elicited~\cite{rucinski1988small}. More details are provided in \cref{proof-motif-density-limit}. 

\subsection{Central limit result for rooted densities}
\cref{motif-density-limit} shows that, conditionally on $x_i$, $s_i(F,G_n)$ is an unbiassed and asymptotically Normal estimator of $s_{x_i}(F,\kappa)$. However, \cref{motif-density-limit} leaves the joint behavior of the rooted densities unexplained. We now consider to what extent it is possible to use the $s_i(F,G_n)$ to estimate distributional properties of $s_{x_i}(F,\kappa)$. The end goal is, for a covariate $y_i$ in some set $\mathcal{Y}$ and observed at each vertex $i$, to use the $s_i(F,G_n)$ to estimate first $\cov\!\big(s_{x_i}(F,\kappa),y_i\big)$, and then a regression, or some other model, linking $s_{x_i}(F,\kappa)$ to $y_i$.

To link $(s_{x_i}(F,\kappa))_{i\in[n]}$ to $Y=(y_i)_{i\in[n]}$, we consider the average across vertices of functions of rooted counts; i.e., for some map $f$, estimators of the form $n^{-1}\sum_i f(s_i(F,G_n),y_i)$. While linking $y_i$ directly to $x_i$ could seem better, that would not make an explicit link between covariates and observed features, making such an approach hard to interpret.

Counterintuitively, the limit distribution of such estimators does not follow immediately from \cref{motif-density-limit}, even for a simple map $f$. Indeed, as for any $F$ we have $\cov\!\big(s_i(F,G_n),s_j(F,G_n)\big)=\Omega\big(n^{-1}\big)$, the dependence between the summands is too strong for a central limit theorem (CLT) to immediately hold; the level of dependence between the $s_i(F,G_n)$ violates what are presented as the best possible sufficient assumption for an exchangeable sequence to verify a CLT~(\citet[p227]{blum1958central}.)

Furthermore, methods to address such dependent sums fail: CLT for $U$-statistics, as in \cite{BickelLevina2012}, apply only if $f$ is a projection; martingale results for exchangeable arrays, such as~\citet{weber1980martingale}, fail because the dependence between the $s_i(F,G_n)$ is too strong; CLT for weakly dependent stationary sequences (e.g.,~\citet{Newman81}), require a specific dependence structure between the $s_i(F,G_n)$ which need not apply for general $\kappa$.

Nonetheless, results such as \cite{BickelLevina2012} and \cite{coulson2016poisson}, show that for $f$ a projection, the estimator is well behaved. Therefore, the covariance between the $s_i(F,G)$ is such that some concentration can occur. Two observations clarify necessary conditions on $f$. First, for the estimator to be consistent, $f$ must be continuous, and since the $s_{x_i}(F,\kappa)$ live in a compact, $f$ can be understood to be uniformly continuous. Then, for $\cov\!\big(f(s_i(F,G), y_i),f(s_j(F,G), y_j)\big)$ to commensurate with $\cov\!\big(s_i(F,G), s_j(F,G)\big)$, $f$ must have an almost everywhere bounded derivative (e.g., see~\citet{cuadras2002}). Therefore, $f$ Lipschitz in its first argument is a necessary assumption, which we find to be sufficient (see \cref{ProofCLT}.)

\begin{restatable}[Central limit result for rooted densities]{thm}{thmglobal}
\phantomsection\label{motif-density-clt}%%%
\vspace{-.5\baselineskip}\hspace{-.31cm}
Let $\big(G_n, Y\big)$ be such that $(G_n)\!\sim\!G(\rho,\kappa)\!$ and $(x_i, y_i)$ is i.i.d. $D$ for some distribution $D$ over $[0,1]\times\mathcal{Y}$. Set the parameter $\gamma\!=\!\max\{e(H)/(|H|\!-\!1)\!:\!v\!\not\in\! H\subset F, |H|\!>\!1\}$ and $f:\smash{[0,+\infty[\times\mathcal{Y}}\mapsto\smash{\mathbb{R}^d}$, Lipschitz in its first argument. Then, if $n\smash{\rho_n^{\gamma}}\to\infty$, $\var s_x(F,\kappa) > 0$, and $f(s_x(F,\kappa),y)\,\vert\,(x, y) \sim D$ has second moments,
\vspace{-.1\baselineskip}\[
\sqrt{n}\left(
\frac1n\sum_{i\in[n]}f\big(s_i(F,G_n),y_i\big)
-
\E f\big(s_i(F, G_n), y_i\big)
\right)
\xrightarrow[n\rightarrow\infty]{L}
\mathrm{Normal}\left(0,\Sigma\right),
\vspace{-.1\baselineskip}\]
where $\Sigma=\cov_{(x,y) \sim D}f\big(s_x(F, \kappa), y\big)$. The result still holds if $s_i(F,G_n)$ is replaced by a vector of rooted densities, and admits a natural Berry-Esseen type extension (\cref{berry-eseen}.)
\vspace{-.5\baselineskip}
\end{restatable}

Fixing $f$ to be the identity function, and omitting $Y$, we observe that \cref{motif-density-clt} generalizes~\citet[Theorem~2]{rucinski1988small},~\citet[Theorem~1]{BickelLevina2012}, \citet[Propositions~35--36]{bordenave2015backtrack} and~\citet[Corollary~4.1]{coulson2016poisson}. \cref{motif-density-clt} also specifies these results by writing the limiting covariance matrix $\Sigma$ as a function of the kernel $\kappa$ and the subgraph $F$ (see \cref{var-closed-form}). The Berri-Essen extension recovers the same rates as presented in~\citet[Theorem~4.2]{privault2020}. See \cref{thm_extensions} for comparison with \cite{barbour2019}.

The assumption underlying \cref{motif-density-clt} is stronger than the one behind \cref{motif-density-limit}. Furthermore, following~\cite[Theorem~2]{rucinski1988small} and~\cite[Theorem~1]{BickelLevina2012}, it is necessary for $\sqrt{n}$-consistency even for $f$ a projection. Also, for some $F$, including trees and cycles, $\gamma=m$ (see the notion of balanced graphs~\cite{erdos59:_on_random_graphs,rucinski1986balanced}).

\cref{motif-density-clt} enables many tests for dependence between rooted counts and vertex covariates. We now turn to use this observation in \cref{Method}, where we will uncover the link between rooted counts and gender in a social network.

\begin{Remark}\label{ml-estimation}
\cref{motif-density-clt} shows that when using $(y_i, s_i(F,G))_i$ to fit a probabilistic model, statistical inference proceeds as if one were using the i.i.d. sample $(y_i, s_{x_i}(F,\kappa))_i$. To see this, fix a smooth log-likelihood function $l_\theta(\cdot,\cdot)$, and set $\hat\theta_G=\arg\!\max_\theta\sum_il_\theta(Y_i, s_i(F,G))$, $\hat\theta_\kappa=\arg\!\max_\theta\sum_il_\theta(Y_i, s_{x_i}(F,\kappa))$. Then, \cref{motif-density-clt} shows that $\hat\theta_G$ and $\hat\theta_\kappa$ converge at the same rate and toward the same value. Further, the asymptotic variance of $\hat\theta_G$ may be computed in the standard way, using the Fisher information matrix. Other approaches, such as, but not limited to, least square optimization and generalized method of moment, would also apply without adaptation. See \cref{uniform-motif-density-clt} for a discussion of the assumptions on $l_\theta$ necessary for this claim to hold.
\end{Remark}

\begin{Remark}\label{np-estimation}
\cref{motif-density-clt} allows for non-parametric estimation of $\kappa$. Specifically, as kernels that are fully determined by a finite number of subgraph densities---also termed {\em finitely forcible} kernels~\cite{lovasz2012large} (e.g., blockmodels)---admit global subgraph densities as sufficient statistics~\cite{ambroise2011new,BickelLevina2012}, they also admit rooted densities as sufficient statistics. Further, following~\cite{lovasz2014automorphism}, we conjecture that rooted densities are in fact sufficient statistics for all kernels of finite rank.
\end{Remark}

%%%%%%%%%%%%%%%%%%%%%%%%%%%%%%%%%%%%%
\section{Methodology: goodness-of-fit and regression}%%%%%%%
\label{Method}%%%%%%%%%%%%%%%%%%%%%%%%%%%%%%
%%%%%%%%%%%%%%%%%%%%%%%%%%%%%%%%%%%%%
We study a social network through two new methods for statistical inference on network data. The first method uses \cref{motif-density-limit} to build a goodness-of-fit test able to locate which vertices cause the lack of fit. The second uses \cref{motif-density-clt} to estimate a regression model linking vertex covariates and rooted densities, and test for significant parameters. Simulations describing the small sample convergence of the proposed methods are presented in \cref{metho}.

\subsection{The dataset}
The network we study, henceforth $G$, is taken from the US National Longitudinal Study of Adolescent Health, specifically school 44 in~\cite{adhealth} which was studied by~\cite{olhede2013network}. Each vertex in $G$ is a student, and a pair of students is connected if at least one of them nominated the other as friend. For almost all students, we have three categorical covariates: gender, race and school grade. There are $n=1117$ students in $G$, comprising $559$ males, $556$ females, $429$ blacks, $525$ whites, $153$ other races. Furthermore, there are $209, 242, 237, 161, 137, 131$ grade $7$--$12$ students respectively. We found no statistically significant correlations between any pair of these covariates. The maximal degree is $32$, while the average degree is $10$ and the median degree is $9$.

To determine whether the local structure of $G$ is a predictor of gender and race, we test whether students' rooted counts are significant predictors of the covariates available. We focus on $\TexttriangleR\,$ and $\!\TextsquareR\ $, and thus use as predictor $\big(s_i(G)\big)_{i\in G} = \smash{\big(s_i(\TexttriangleR\ ,G),s_i(\!\TextsquareR\ ,G)\big)_{i\in G}\in\mathbb{R}^{n\times 2}}$. We choose $\TexttriangleR\,$ and $\!\TextsquareR\ $ as: using the smaller subgraphs presented little power; $\TexttriangleR\,$ and $\!\TextsquareR\ $ are not subgraphs of each other, which we observed reduces the correlation of the counts; they are prominently used to characterize local network properties~\cite{przulj2007,ali2014alignment}.

\subsection{Goodness-of-fit test} 
Before analyzing $G$ at the level of each student---through using the $s_i(G)$ to predict the covariates---we first test for $G$ to be a realization of an inhomogeneous random graph. Indeed, while $G$'s global features of sparsity, community structure and degree distribution have successfully been modeled by our null model~\cite{olhede2013network}, our results will now allow us to test if our null is also sufficient to account for local features of $G$. We do so by testing independently for each $i\in G$ that $s_i(G)$ is consistent with a kernel estimate. We choose such a vertex-specific approach, instead of a global one relying on averages over vertices, for two key reasons: first, a global test would be less powerful, as averages of rooted count are asymptotically normal under other, more general, null models~\cite{coulson2016poisson, orbanz2017sub}; second, this will enable us to determine which vertices cause the lack of fit.

Out goodness-of-fit test takes places in three steps: i) the kernel and latent variates are estimated using maximum likelihood, yielding $\hat\kappa$ and $(\hat x_i)_i$; ii) we use Monte-Carlo integration methods with $\hat\kappa$ and $(\hat x_i)_i$ to compute the mean $\hat \mu_i$ and variance $\hat\sigma_i$ of each $s_i(G)$ under the null; iii) we test across $i$ that the $\hat t_i = \smash{\hat\sigma_i^{_{-1/2}}(s_i(G)-\hat \mu_i)}$ are realizations of the bivariate standard Normal distribution. Note that the approach only makes sense since we know from~\cref{motif-density-limit} that all the $t_i$ tend to the same distribution. Further details, especially on the estimation of $\kappa$, are provided in \cref{metho}. Two key points are that: the plug-in estimators $\hat \mu_i$ and $\hat\sigma_i$ may be used for standardization because the variation of the rooted densities is at least as large as that of $\hat\kappa$ and $(\hat x_i)_i$~\cite{lei2015,chatterjee2015,sarkar2015}; the dependence between the $s_i(G)$ is such that using Bonferroni correction yields a too conservative test, and we present an alternative correction in \cref{metho}.

\begin{figure}[t]
\centering
\vspace{-\baselineskip}
\includegraphics[width=.6\textwidth]{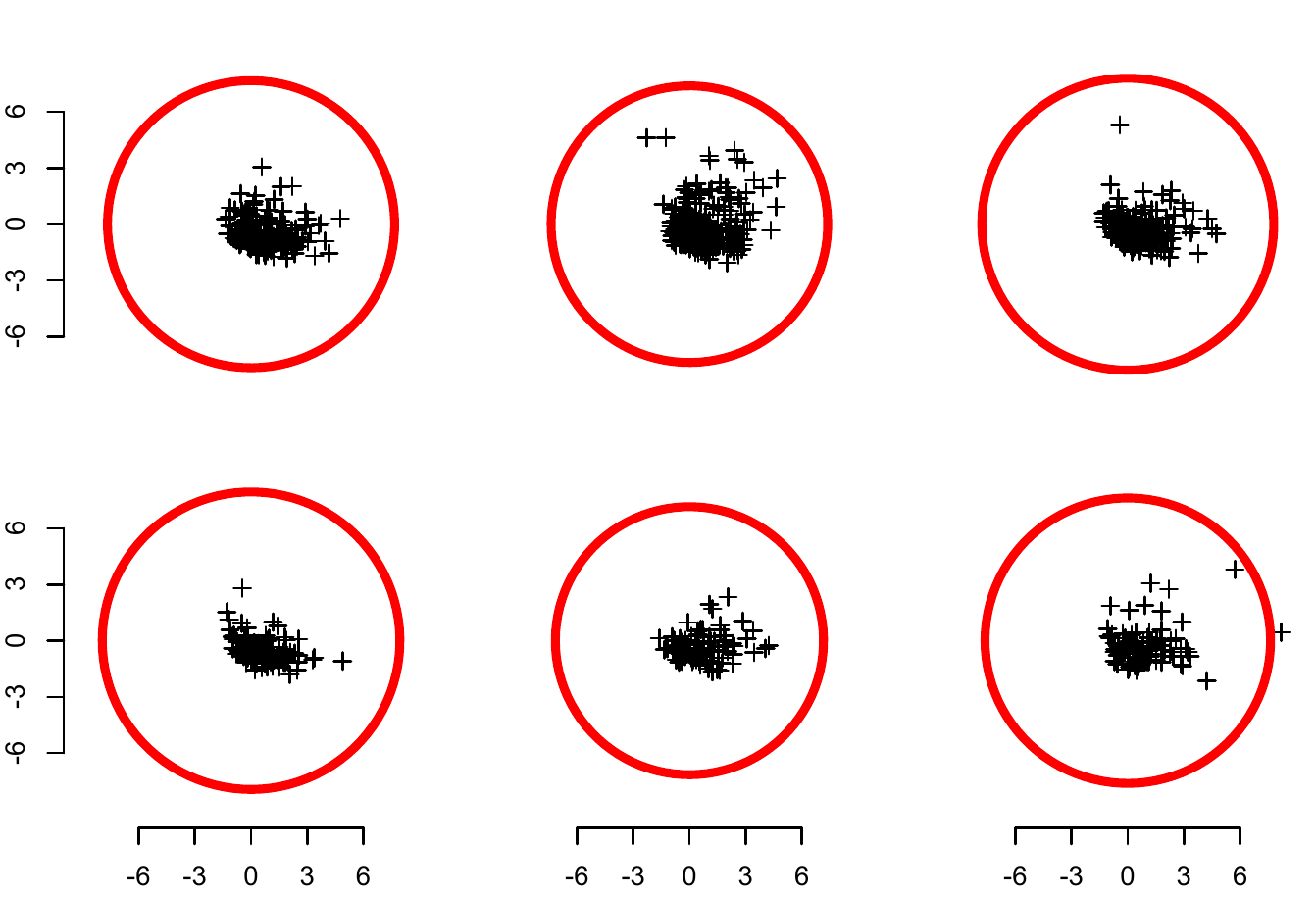}
\caption{Goodness-of-fit within grades. Crosses are the standardized rooted density vectors $\smash{\hat t_i(G_g)}$, where from top to bottom, and left to right, $g$ ranges $7$--$12$. The red circle is the confidence region at the $10\%$ level (see \cref{metho} for a description of our algorithm to estimate the radius of this circle). Only $G_{12}$ is rejected, with one outlier vertex.
\vspace{-\baselineskip}}
\label{Sim}
\end{figure}

Using this method we reject the hypothesis that $G$ is a realization of an inhomogeneous random graph: $43$ $t_i$-s are too far from the origin. However, we note that the vertices causing the lack of fit have more edges spanning grades than most. Thus, we refine our analysis to the graphs formed by students in the same grade; the $G_g$ for $g\in\{7,\dots,12\}$. There, using the same method we reject the null only for $G_{12}$ (see \cref{Sim}).

\subsection{Regression} We now use the the rooted densities to predict gender and race. Specifically, following the outcome of the tests, we consider as predictor $\big(s_i(G_{g(i)})\big)$, where $i$ ranges the students in grade less than $11$, and $g(i)$ is the grade of student $i$. Then, we predict the gender and race covariate via a binomial linear regressions with a logit link function. Following \cref{ml-estimation}, this model may be estimated by maximum likelihood and analyzed in the standard fashion. There, we can make the following observations (always at the $10\%$ level): %
i) While white students do not have an average degree larger than black students, they present significantly larger rooted triangle and square densities; %
ii) By conditioning on race, in both cases we find that the likelihood of a student being female increases with the triangle density and shrinks with the squares density, while for male students, it is the reverse; %
iii) The triangle and square densities do not evolve with grade. These findings are robust to removing the vertices causing the lack of fit in $G$.

We note that when using $\big(s_i(G_{g(i)})\big)$ in a regression model, we do not rely on $\hat\kappa$ and $(\hat x_i)_i$. This makes the approach independent of any kernel estimate, vertex clustering, or group label assignment, but also requiring only negligible computation beyond that of evaluating the rooted counts.

To conclude, we have shown that: i) the whole network cannot be modeled by our null; ii) yet, local network structure changes with vertex covariate in a provably predictable way. Expanding our analysis to more social networks would enable rich and general inference regarding the scale at which conditional independence is a relevant null, as well as the link between local network structures and gender and race, ultimately allowing to explore the global consequence of group specific behavior on the social graph.

%%%%%%%%%%%%%%%%%%%%%%%%%%%%%%%%%
\section{Discussion}\label{Discussion}%%%%%%%%%%%%%%
%%%%%%%%%%%%%%%%%%%%%%%%%%%%%%%%%

We contribute the results, and associated methods, that enable the study of large networks at the level of each vertex. This in the exact same range of density regimes as group level inference is possible. The combinatorial proof techniques introduced constitute a basis opening new venues for the study of local properties of random graphs. We expect that our proofs extend to: directed or weighted graphs, unbounded kernels as long as $s_U(F,\kappa)$ exists, as well as to rooted subgraphs with multiple roots. 

\begin{figure}[t]
	\centering
	\setlength{\mylength}{\linewidth}
	\scalebox{.5}{\input{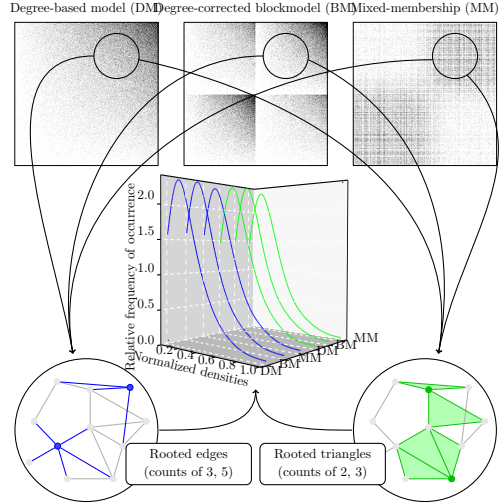}}%
	\caption{Three different kernels that lead to similar rooted edge and triangle densities. We consider three networks ensuing from three kernels and having substantially different community structures: no communities, sharply separated communities, and overlapping communities. After normalizing the rooted densities to account for their different scales, we observe that the six obtained distributions are very similar across kernels and densities.}
	\label{ComunityMotif}
\end{figure}

Following our findings, it appears that inhomogeneous random graphs---which are recognized as sufficiently flexible to fully account for the three key large scale features of real-world networks that are sparsity, community structure and heavy tailed degree distribution~\cite{olhede2013network}---do not induce rich local behaviors. Indeed, after standardization, rooted densities are almost i.i.d normal, both across vertices and subgraphs. This observation parallels the results presented in~\cite{chatterjee2013, Seshadhri5631}. We explore the idea further in \cref{ComunityMotif}, where we show that rooted densities density, and community structure can be, to a large extent, dissociated under our null model.

This makes our methods especially useful, because they therefore uncover a new layer of information to be extracted from network data. However, it also strains the modeling capabilities of the null model. The random graph models discussed in~\cite{Miller2009percolation,Newman2009clustering,Karrer2010subgraphs,Bollobas2011clustering} generate richer local features. However, these models all take local counts as input, and thus cannot be used to infer the links between local counts and vertex covariates. Therefore, the problem of proposing a framework able to induce rich local behavior remains open.

%%%%%%%%%%%%%%%%%%%%%%%%%%%
\appendix%%%%%%%%%%%%%%%%%%%%%%
%%%%%%%%%%%%%%%%%%%%%%%%%%%
\makeatletter
\renewcommand{\@seccntformat}[1]{
	APPENDIX~{\csname the#1\endcsname}:\hspace*{1em}}
\makeatother
\setcounter{Theorem}{0}%
\renewcommand{\theTheorem}{\thesection.\arabic{Theorem}}
\setcounter{Remark}{0}%
\renewcommand{\theRemark}{\thesection.\arabic{Remark}}
\setcounter{Lemma}{0}%
\renewcommand{\theLemma}{\thesection.\arabic{Lemma}}
\setcounter{Proposition}{0}%
\renewcommand{\theProposition}{\thesection.\arabic{Proposition}}
\setcounter{Definition}{0}%
\renewcommand{\theDefinition}{\thesection.\arabic{Definition}}
\setcounter{Corollary}{0}%
\renewcommand{\theCorollary}{\thesection.\arabic{Corollary}}
\setcounter{figure}{0}%
\renewcommand{\thefigure}{\thesection.\arabic{figure}}
\makeatletter
\renewcommand\subsection{\@startsection {subsection}{1}{\parindent}%
	{\medskipamount}%
	{-10pt}%
	{\subsection@shape}}
\def\subsection@shape{\normalsize\itshape}
\def\subsection@prefix{\normalshape}
\makeatother
%%%%%%%%%%%%%%%%%%%%%%%%%%%
\newcommand{\ThmLoc}{\ref{motif-density-limit}}%%%
\section{Proof of Theorem~\protect\ThmLoc{}}%%%%
\label{proof-motif-density-limit}%%%%%%%%%%%%
%%%%%%%%%%%%%%%%%%%%%%%%%%%

%%%%%%%%%%%%%%%%%%%%%%%%%%%
\subsection{Number of rooted copies}
%%%%%%%%%%%%%%%%%%%%%%%%%%%
In a first instance, the core object of our study will be the number of copies of a rooted graph within a larger graph. Therefore, for a fixed rooted graph $F=[F,v]$ and a graph $G$, we first introduce the number of rooted isomorphic copies of $F$ with root at a fixed vertex in $G$, and provide examples in \cref{X-F-examples}.

\begin{Definition}[Number of rooted isomorphic copies]
Fix a rooted graph $F=[F,v]$ and a graph $G$. We denote $X_F(G,i)$ the number of rooted isomorphic copies of $F$ rooted at $i\in G$; i.e,
\[
X_F(G,i) = \#\left\{F'\subset G : i\in F' \text{ and } [F',i]\equiv F\right\},
\]
with $[F',i]\equiv F$ if there is an adjacency preserving isomorphism from $F'$ to $F$ that maps $i$ to $v$, and $F'\subset G$ spans the non-automorphic subgraphs of $G$.
\end{Definition}

As we count non-automorphic subgraphs of $G$, we have that $X_F(K_n,i) = (n-1)_{|F|-1}/\mathrm{aut}(F)$, with $K_n$ the complete graph, $(n)_k = n(n-1)\cdots(n-k+1)$ the falling factorial, and $\mathrm{aut}(F)$ the order of the automorphism group of $F$ (under $\equiv$.) The overall number of copies of $F$ in $G$, usually denoted $X_F(G)$, is then proportional to the sum over $i\in G$ of the $X_F(G,i)$.

To begin our analysis we consider how the product of the number of rooted isomorphic copies of two rooted graphs $F_1$ and $F_2$ behaves. Therefore, we evaluate $X_{F_1}(G,i)X_{F_2}(G,i)$ and write
\begin{align}
\nonumber
X_{F_1}(G,i)X_{F_2}(G,i)
&= \sum_{i\in F_1'\subset G}\II{\{[F_1',i]\equiv F_1\}}
\sum_{i\in F_2'\subset G}\II{\{[F_2',i]\equiv F_2\}}\\
\label{rooted-copies-product-1}
&= \sum_{\substack{i\in F_1'\subset G\\i\in F_2'\subset G}}\II{\{[F_1',i]\equiv F_1\}}\II{\{[F_2',i]\equiv F_2\}}.
\end{align}
To reformulate this sum in a more concise way, we index all the unlabeled rooted graphs that can be formed by two graphs $F_1'$ and $F_2'$ containing $i$ and such that $[F_1',i]\equiv F_1$ and $[F_2',i]\equiv F_2$. This yields the following definition and subsequent lemma.

\begin{figure}[t]
{\centering{
\begin{tikzpicture}[auto,thick, scale=.75]
		\tikzstyle{pop}=[circle, color=black,fill=black!30,
						inner sep=0.1pt, minimum width=10pt]
		\tikzstyle{popi}=[circle, draw, color = black, fill=black,
						inner sep=0pt, minimum width=5pt]
		\tikzset{pil/.style={very thick,color=black,shorten >=2.25pt}}
		%% Grey Edges and nodes
		\draw {(0:0) node[popi] {} edge[pil] (90:2) node[popi] {}};
		\draw {(90:2) node[popi] {} edge[pil] (90+36:3) node[popi] {}};
		\draw {(0:0) node[popi] {} edge[pil] (330:2) node[popi] {}};
		\draw {(330:2) node[popi] {} edge[pil] (30:1.5) node[popi] {}};
		\draw {(330:2) node[popi] {} edge[pil] (330-36:3) node[popi] {}};
		\draw {(0:0) node[popi] {} edge[pil] (330-36:3) node[popi] {}};
		\draw {(90:2) node[popi] {} edge[pil] (30:1.5) node[popi] {}};
		\draw {(210-36:3) node[popi] {} edge[pil] (90+36:3) node[popi] {}};
		\draw {(210+36:3) node[popi] {} edge[pil] (330-36:3) node[popi] {}};
		\draw {(30:1.5) node[popi] {} edge[pil] (0:0) node[popi] {}};		
		\draw {(0:0) node[popi] {} edge[pil] (210:2) node[popi] {}};
		\draw {(36+210:3) node[popi] {} edge[pil] (210:2) node[popi] {}};
		\draw {(210-36:3) node[popi] {} edge[pil] (210:2) node[popi] {}};
		\draw {(210:3.75) node[popi] {} edge[pil] (210:2) node[popi] {}};
		\draw {(330-36:3) node[popi] {} edge[pil] (210:2) node[popi] {}};
		\draw{(210:2) node[pop] {$j$}};
	
		\draw {(45:3) node[popi] {} edge[pil] (90:2) node[popi] {}};
		\draw {(90+36:3) node[popi] {} edge[pil] (45:3) node[popi] {}};
		\draw {(330:2) node[popi] {} edge[pil] (45:3) node[popi] {}};
		\draw{(45:3) node[pop] {$i$}};
		\draw{(90:2) node[popi] {}};
	\end{tikzpicture}
	\caption{Example of rooted graph counting. Call $G$ the plotted graph, then we have:
	$X_{\protect\TexttriangleR}\,(G, i) = 1$,
	$X_{\protect\TextcherryR}\,(G, i) = 8$,
	$X_{\protect\TextdiamondR}\,(G, i) = 0$,
	$X_{\protect\TextsquareR}\ (G, i) = 2$,
	$X_{\protect\TexttriangleR}\,(G, j) = 2$,
	$X_{\protect\TextcherryR}\,(G, j) = 9$,
	$X_{\protect\TextdiamondR}\,(G, j) = 1$,
	$X_{\protect\TextsquareR}\ (G, j) = 2$.
	\label{X-F-examples}}}}
\end{figure}

\begin{Definition}[Overlapping copies~\protect\citep{rucinski1988small}]
\phantomsection\label{overlapping-copies}
For two rooted graphs $F_1$ and $F_2$, we write $\mathcal{H}_{F_1,F_2}$ the set of unlabeled rooted graphs that can be formed by two copies of $F_1$ and $F_2$, and $c_H$ the number of ways a given $H=[H,v]\in\mathcal{H}_{F_1,F_2}$ can be built from copies of $F_1$ and $F_2$:
\[\begin{dcases}
\mathcal{H}_{F_1,F_2} &= 
	\left\{[H,v]\subset K_{|F_1|+|F_2|}\,: 
		\parbox{.42\textwidth}{\centering
			$\exists [F'_1,v],[F'_2,v]\subset K_{|F_1|+|F_2|}$ s.t.,\\[.1\baselineskip]
			$[F_1',v]\equiv F_1, [F'_2,v]\equiv F_2$,\\[.1\baselineskip]
			and $H = F_1'\cup F_2'$}\right\}/\equiv\\
c_H &= \#\left\{\big([F_1',v],[F_2',v]\big)\subset [H,v]\,:
	\parbox{.33\textwidth}{\centering
		$[F_1',v]\equiv F_1, [F_2',v]\equiv F_2$\\[.1\baselineskip]
		and $H= F_1'\cup F_2'$}\right\}.
\end{dcases}\]
\end{Definition}

If $H$ is obtained from copies of $F_1$ and $F_2$ overlapping only at the root, then $c_H = \mathrm{aut}(H)/\mathrm{aut}(F_1)\mathrm{aut}(F_2)$~\cite[p4]{rucinski1988small}. No formula exists in the general case; see \cref{c-H-examples}.

\begin{figure}[t]
\begin{center}\begin{equation*}\begin{array}{rlll}
\mathcal{H}_{\TexttriangleR\ ,\,\TexttriangleR}\!\!\!&= \left\{\!
	\TextbowtieR\, ,
	\TextdiamondR\ ,
	\TexttriangleR\,\right\};&\!\!\!
	c_{\TextbowtieR}(\TexttriangleR\ ,\TexttriangleR\,)=2,&\!
	c_{\TextdiamondR}\,(\TexttriangleR\ ,\TexttriangleR\,)=2,\\&&\!\!\!
	c_{\TexttriangleR}\,(\TexttriangleR\ ,\TexttriangleR\,)=1.&\!\\[1.5\baselineskip]
	\ &\cline{1-2}&\ \\[-.2cm]
\mathcal{H}_{\TextcherryR\ ,\,\TextcherryR}\!\!\!&= \left\{\!
	\TextDoubleTwoPathsR\, ,
	\TextsquareR\ ,
	\TextshovelR\ ,
	\TexttrypodR\ ,
	\TexttriangleR\ ,
	\TextcherryR\,\right\};&\!\!\!
	c_{\TextDoubleTwoPathsR}(\TextcherryR\ ,\TextcherryR\,) = 2,&\!
	c_{\TextsquareR}\ (\TextcherryR\ ,\TextcherryR\,) = 2,\\&&\!\!\!
	c_{\TextshovelR}\ (\TextcherryR\ ,\TextcherryR\,) = 2,&\!
	c_{\TexttrypodR}\ (\TextcherryR\ ,\TextcherryR\,) = 2,\\&&\!\!\!
	c_{\TexttriangleR}\,(\TextcherryR\ ,\TextcherryR\,) = 2,&\!
	c_{\TextcherryR}\,(\TextcherryR\ ,\TextcherryR\,) = 1.\\
	\ &\cline{1-2}&\ \\[-.2cm]
\mathcal{H}_{\TexttriangleR\ ,\,\TextcherryR}\!\!\!&= \left\{\!
	\TextTriangleTwoPathsR\, ,
	\TextdiamondR\ ,
	\TextshovelR\ ,
	\TexttriangleR\,\right\};&\!\!\!
	c_{\TextTriangleTwoPathsR}(\TexttriangleR\ ,\,\TextcherryR\,) = 1,&\!
	c_{\TextdiamondR}\,(\TexttriangleR\ ,\TextcherryR\,) = 2,\\&&\!\!\!
	c_{\TextshovelR}\ (\TexttriangleR\ ,\TextcherryR\,) = 1,&\!
	c_{\TexttriangleR}\,(\TexttriangleR\ ,\TextcherryR\,) = 2.\!
\end{array}\end{equation*}\end{center}
\caption{Examples of $\mathcal{H}_{F_1,F_2}$ and $c_H$. We add only within this figure the dependence of $c_H$ in $F_1, F_2$. The first element of each $\mathcal{H}_{F_1,F_2}$ is $F_1F_2$ (see \cref{gluing-product}). The $c_H$ are obtained by counting the number of ways $H$ may be covered by one rooted copy of $F_1$ and $F_2$ each. The $c_H$ may also be obtained inductively by using the implicit preorder in $\mathcal{H}_{F_1, F_2}$ as follows: i) write $F_1\subset F_2$ if $X_{F_1}(F_2)>0$; ii) for the smallest elements of the poset $(\mathcal{H}_{F_1, F_2}, \subset)$, say the $H_{1t}$, we have $c_{H_{1t}} = X_{F_1}(H_{1t})X_{F_2}(H_{1t})$ by \cref{rooted-copies-product}; iii) for the smallest elements of $(\mathcal{H}_{F_1, F_2}\setminus\{H_{1t}\}_t, \subset)$, say the $H_{2t}$, we have $c_{H_{2t}} = X_{F_1}(H_{2t})X_{F_2}(H_{2t}) - \sum_s c_{H_{1s}}X_{H_{1s}}(H_{2t})$ again by \cref{rooted-copies-product}; and so on until termination.\label{c-H-examples}}
\end{figure}

\begin{Lemma}[Rooted copies pairwise interaction]
\phantomsection\label{rooted-copies-product}
Fix two rooted graphs $F_1,F_2$ and a graph $G$. Then, for any $i\in G$,
\[
X_{F_1}(G,i)X_{F_2}(G,i) = 
\sum_{H\in\mathcal{H}_{F_1,F_2}}c_HX_H(G,i).
\]\vspace{-\baselineskip}
\end{Lemma}
\begin{proof}
This proofs consists in rewriting~\eqref{rooted-copies-product-1} using \cref{overlapping-copies}. We first note that for each $F_1',F_2'$ inf~\eqref{rooted-copies-product-1}, $\II{\{[F_1',i]\equiv F_1\}}\II{\{[F_2',i]\equiv F_2\}}=1$ if and only if there exists a rooted graph $H\in\mathcal{H}_{F_1,F_2}$ such that $[F_1'\cup F_2',i] \equiv H$. Therefore we can reindex the sum in~\eqref{rooted-copies-product-1} as follows:
\begin{align*}
X_{F_1}(G,i)X_{F_2}(G,i)
&= \sum_{i\in F_1',F_2'\subset G}\II{\{\exists H\in\mathcal{H}_{F_1,F_2}\,:\, [F_1'\cup F_2',i]\equiv H\}}\\
&= \sum_{H\in\mathcal{H}_{F_1,F_2}}\ \sum_{i\in F_1',F_2'\subset G}\II{\{[F_1'\cup F_2',i]\equiv H\}}.
\end{align*}
We now note that by definition of $c_H$, for each copy of $H$ in $G$ rooted at $i$, there will be $c_H$ pairs $F_1',F_2'$ of copies of $F_1$ and $F_2$ rooted at $i$ in $G$ such that $[F_1'\cup F_2',i] = H$. Therefore, we can simplify the sum above to obtain,
\begin{align*}
X_{F_1}(G,i)X_{F_2}(G,i)
&= \sum_{H\in\mathcal{H}_{F_1,F_2}} c_H\sum_{i\in H'\subset G}\II{\{[H',i]\equiv H\}}\\
&= \sum_{H\in\mathcal{H}_{F_1,F_2}} c_HX_H(G,i),
\end{align*}
yielding the sought-after result.
\end{proof}

A key consequence of \cref{rooted-copies-product} is that the quantities $X_F(G,i)$ for all $F$ are sufficient to express products of the form $X_{F_1}(G,i)X_{F_2}(G,i)$. Indeed, although no simple formula exists to compute $\mathcal{H}_{F_1,F_2}$~\citep{rucinski1988small}, it is a functions only of $F_1$ and $F_2$, but not of $G$. Then, through an induction argument, the quantities $X_F(G,i)$ for all $F$ are sufficient to express higher order products of the form $X_{F_1}(G,i)\cdots X_{F_m}(G,i)$ for some $m>1$. As we shall see below, in the case where $G$ is a random graph, a similar phenomena occurs: describing the first moment of $X_F(G,i)$ for all $F$ is sufficient to describe the higher order moments of the rooted counts, and ultimately the joint distribution of the rooted graph counts.

%%%%%%%%%%%%%%%%%%%%%%%%%%%%%%%%%
\subsection{Rooted copies in inhomogeneous random graphs}%%
%%%%%%%%%%%%%%%%%%%%%%%%%%%%%%%%%
We now consider $X_{F}(G_n,i)$ in the setting where $G_n$ is a sequence of growing inhomogeneous random graphs (see \cref{model}). Our focus will be the limiting distribution of $X_F=X_F(G_n, i)$. 

In what follows, we first uncover the limiting distribution of $X_F$ conditionally on $x_i=x$ for some $x\in[0,1]$. Therefore, we will always implicitly consider counts in $G_n$ rooted at some vertex $i$, and $(G_n)$ as being a sequence of inhomogeneous random graphs where the latent variable $x_i$ is fixed to be equal to a given value $x\in[0,1]$; i.e. let $(G_n)\sim G(\rho,\kappa)|x_i=x$. Furthermore, we will omit, unless it is necessary, the dependence in $n$ of both $G_n$ and $\rho_n$.

We begin by presenting the first order properties of $X_F$, the proof of which is deferred to the \cref{proof-of-first-order}, where we also present slightly improved bounds:
\begin{Proposition}\phantomsection\label{first-order}
With $F$ a rooted graph and $G\sim G(\rho,\kappa)|x_i=x$,
\begin{align*}
\E X_F&=s_x(F,\kappa)\,\rho^{e(F)}X_F(K_n,i),\\
\var X_F&=\big(1+O(1/n)\big)\textstyle\sum_{H\in\mathcal{H}_{F,F}\setminus\{F^2\}} c_H\E X_H,
\end{align*}
where $s_x(F,\kappa) = \E \big[\prod_{pq\in F}\kappa(x_p,x_q)\,\vert\, x_i=x\big]$, and $F^2$ is the rooted graph obtained by taking the disjoint union of two copies of $F$ and identifying their roots. Then, with $m(F)=\max\{e(H)/(|H|-1)\,:\,v\in H\subset F, |H|>1\}$, if $n\rho^{m(F)}=\omega(1)$, we have that
\begin{equation}\label{hom-prop-eq-eq}
X_F = \E X_F\left(1+O_p\big(1/n\rho^{m(F)}\big)\right).
\end{equation}
Furthermore, for $F_1,F_2$ two labelled rooted graphs sharing their root, we have
\begin{equation}\label{motif-set-operations}
X_{F_1\cup F_2}=O_p\big(\E X_{F_1}\E X_{F_2}/\E X_{F_1\cap F_2}\big).
\end{equation}
\end{Proposition}

\begin{Remark}[Computing the variance]\label{var-closed-form}
Although it is less amenable to the subsequent derivation, we note that:
\[
\var X_F = \big(1+O(1/n)\big) \sum_{H\in\mathcal{H}_{F,F}\setminus\{F^2\}}\ s_x(H,\kappa)\,\rho^{e(H)}\frac{(n)_{|H|-1}}{\aut(H)}.
\]
\end{Remark}
We now show how \cref{first-order}, with some additional manipulations, may be augmented to provide a description of any central moment of $X_F$. In turns---through method of moment arguments---we obtain a more precise understanding of the joint behavior of the $X_F$ in the limit of large graphs.

\begin{Theorem}\phantomsection\label{cross-mom}
Fix $k$ rooted graphs $F_1,\dots,F_k$. Let $G\sim G(\rho,\kappa) | x_i = x$ and assume that $\epsilon = 1/n\rho^{\max_{t\in[k]}m(F_t)}=o(1)$ and that $s_X(F_t,\kappa)>0$ for $t\in[k]$. Then, if $k$ is even,
\[
\E\prod_{t=1}^k\left(X_{F_t}-\E X_{F_t}\right)
= \big(1+O(\epsilon)\big)
\sum_{s\in [k]^{(2)}}
	\prod_{\{p,q\}\in s}\cov\left(X_{F_p},X_{F_q}\right),
\]
where $[k]^{(2)}$ is the set of all partitions of $[k]$ in subsets of size $2$. On the other hand, if $k$ is odd, $\E\prod_t\big(X_{F_t}-\E X_{F_t}\big) = O\big(\epsilon\prod_t\var(X_{F_t})^{1/2}\big)$.
\end{Theorem}

\begin{proof}
This proof proceeds in three steps. First we develop a more tractable expression of $\E\prod_t(X_{F_t}-\E X_{F_t})$ in the form of a sum mirroring that of \cref{rooted-copies-product}. Then, we determine the leading term in that sum, by determining how the copies of the $F_t$ in $G$ predominantly overlap. Finally, we conclude in the same fashion as in the proof of \cref{first-order}.

\paragraph{Expanding $\E\prod_t(X_{F_t}-\E X_{F_t})$:}
Let $F_{t1},\dots,F_{tM_t}$ be the $M_t = X_{F_t}(K_n,i)$ labeled rooted copies of $F_t$ in $K_n$. Then, $X_{F_t}=\sum_{m=1}^{M_t}\II{\{F_{tm}\subset G\}}$, so that
\begin{align}
\nonumber
\prod_{t=1}^k\left(X_{F_t}-\E X_{F_t}\right)
&=\prod_{t=1}^k\left(\sum_{m=1}^{M_t}\II{\{F_{tm}\subset G\}}-\sum_{m=1}^{M_t}\P[F_{tm}\subset G]\right)\\
\nonumber
&=\prod_{t=1}^k\left(\sum_{m=1}^{M_t}\left(\II{\{F_{tm}\subset G\}}-\P[F_{tm}\subset G]\right)\right)\\
\label{cross-mom-1}
&=\sum_{m_1\in[M_1]}\sum_{m_2\in[M_2]}\cdots\sum_{m_k\in[M_k]}\prod_{t=1}^k \left(\II{\{F_{tm_t}\subset G\}}-\P[F_t\subset G]\right),
\end{align}
where the last step follows form expanding a product of sum into a sum of products, and using that for all $t$ and $m$, $\P[F_{tm}\subset G]=\P[F_t\subset G]$.

We now distinguish between $k$-tuples $(m_1,\dots,m_k)$ according to the unlabelled rooted graph induced by $\cup_t F_{tm_t}$. Therefore, we call $\mathcal{H}$ the set of unlabelled rooted graphs that can be formed by overlapping rooted copies of the rooted graphs $F_1,\dots, F_k$; i.e., generalizing \cref{overlapping-copies} we write
\[
\mathcal{H} = \left\{
[H,v]\subset K_{\sum_t\!|F_t|} : 
	\exists [F'_t,v] \ \mathit{ s.t. }\ \forall t, [F'_t,v]\equiv F_t\ \mathit{ and }\ H=\cup_t F'_t
\right\}/\equiv.	
\]
Furthermore, to simplify notation, for $H\in\mathcal{H}$ we write
\begin{equation}\label{cross-mom-1-5}
\tilde X_H =
\sum_{m_1\in[M_1]}\sum_{m_2\in[M_2]}\cdots\sum_{m_k\in[M_k]}\II\{\cup_tF_{tm_t}\equiv H\}\prod_{t=1}^k \left(\II{\{F_{tm_t}\subset G\}}-\P[F_t\subset G]\right)
\end{equation}
Resuming from~\eqref{cross-mom-1} we may therefore write:
\begin{equation}\label{cross-mom-2}
\E\prod_{t=1}^k\left(X_{F_t}-\E X_{F_t}\right)
 = 
\sum_{H\in\mathcal{H}}\E\tilde X_H.
\end{equation}
Note that the $F_{tm_t}$ forming $H$ may overlap in varied ways. As we eventually bound $\E\tilde X_H$ by $\E X_H$, only the resulting shape $H=\cup_tF_{tm_t}$ matters to our analysis, and how they overlap does not affect the following arguments. This observation is also why we can choose this notation.

\paragraph{Order of magnitude of $\E\tilde X_H$:}
We begin by controlling the order of magnitude of each term in the sum of~\eqref{cross-mom-2}. For each $H = \cup_tF_{tm_t}$, this is done by discussing how many vertices and edges of the $F_{tm_t}$ overlap.

First we note that if there exists $p$ such that $|F_{pm_p}\cap(\cup_{t\neq p}F_{tm_t})|\leq1$, then we must have that $F_{pm_p}\cap(\cup_{t\neq p}F_{tm_t}) = (\{i\},\emptyset)$ and $F_{pm_p}$ has all its vertices except its root disjoint from the $F_{tm_t}$. Since vertex disjoint edges occur independently in $G(n,\rho\kappa)|x_i=x$, we have
\begin{multline}\label{cross-mom-22}
\E\prod_{t=1}^k\left(\II{\{F_{tm_t}\subset G\}}-\P[F_t\subset G]\right)\\
= \left\{\E\prod_{t\neq p} \left(\II{\{F_{tm_t}\subset G\}}-\P[F_t\subset G]\right)\right\}\E\left(\II{\{F_{pm_p}\subset G\}}-\P[F_p\subset G]\right)
=0,\!\!\!\!\!\!
\end{multline}
as $\E\II{\{F_{pm_p}\subset G\}}=\P[F_p\subset G]$ by definition. Therefore, we need only consider $H = \cup_tF_{tm_t}$ such that each $F_{tm_t}$ overlap with at least one other $F_{tm_t}$. In this case, expanding the product using inclusion exclusion, we obtain
\begin{multline}
\label{cross-mom-3}
\E\prod_{t=1}^k\left(\II{\{F_{tm_t}\subset G\}}-\P[F_t\subset G]\right)\\=
\sum_{S\subset [k]} (-1)^{|S|}\E\left[\II{\{\cup_{t\not\in S}F_{tm_t}\subset G\}}\right]\prod_{t\in S}\P[F_t\subset G].
\end{multline}
From \cref{nb-copy-order}, we have that $\E\II{\{\cup_{t\not\in S}F_{tm_t}\subset G\}}=O(\rho^{e(\cup_{t\not\in S}F_{tm_t})})$ and $\prod_{t\in S}\P[F_t\subset G]=O(\rho^{\sum_{t\in S}e(F_t)})$. As, by assumption, we have $|(\cup_{t\not\in S}F_{tm_t})\cap (\cup_{t\in S} F_{tm_t})|>1$, it must be that $e(\cup_{t\not\in S}F_{tm_t})+\sum_{t\in S}e(F_t)\geq e(\cup_tF_{tm_t})$, and since $\rho<1$, we therefore have
\[
\E\II{\{\cup_{t\not\in S}F_{tm_t}\subset G\}}\prod_{t\in S}\P[F_t\subset G]= O\left(\rho^{e(\cup_tF_{tm_t})}\right)=O\left(\rho^{e(H)}\right).
\]
As a result, from~\eqref{cross-mom-1-5}, ~\eqref{cross-mom-3} and \cref{nb-copy-order}, we obtain
\begin{align}
\nonumber
\E \tilde X_H
&= O\left(\rho^{e(H)}\#\big\{(m_1,\dots,m_k)\,:\, \cup_tF_{tm_t}\equiv H\big\}\right)\\
\label{cross-mom-3-2}
&= O\left(\rho^{e(H)}X_H(K_n,i)\right)
= O\left(\E X_H\right).
\end{align}

\paragraph{Controlling the variance:}
Before we proceed to bound $\E X_H$, we first control the standard deviation $\var(X_{F_t})^{1/2}$. There, we note that from \cref{first-order},
\begin{equation}\label{cross-mom-5.5}
\var X_{F_t}
= \big(1+O(1/n)\big) \sum_{U\in\mathcal{H}_{F_t,F_t}\setminus\{F_t^2\}} c_U\E X_U.
\end{equation}
Then, recall from \cref{overlapping-copies} that for each $U$ in~\eqref{cross-mom-5.5} there are two overlapping copies of $F_t$, say $U_1$ and $U_2$, such that $U = U_1\cup U_2$. Thus, we may apply \cref{first-order} again to each $\E X_U$ in the sum above, and reindex the sum by the overlap $J = U_1\cap U_2$, and as $s_x(F_t,\kappa)>0$ gives us that $s_x(J, \kappa)>0$ and therefore $\E X_J>0$, we obtain
\[
\var X_{F_t}
= O\left(\sum_{J\subset F_t\,:\, |J|>1}\frac{\left(\E X_{F_t}\right)^2}{\E X_J}\right),
\]
and, as only a finite number of terms is in the sum and therefore can dominate it, we get
\[
\var(X_{F_t})^{1/2} = 
O\left(\sum_{J\subset F_t\,:\, |J|>1}\frac{\E X_{F_t}}{\sqrt{\E X_J}}\right).
\]
Furthermore, for every non-zero term in the right hand side of~\eqref{cross-mom-5.5}, we have from \cref{first-order} that $\E X_U = \Theta((\E X_{F_t})^2/\E X_J)$. Therefore that for any $J\subset F$ we have:
\begin{equation}\label{cross-mom-6}
\Omega\left(\frac{\E X_{F_t}}{\sqrt{\E X_J}}\right) = \var(X_{F_t})^{1/2} = O\left(\sum_{J\subset F_t\,:\, |J|>1}\frac{\E X_{F_t}}{\sqrt{\E X_J}}\right).
\end{equation}

\paragraph{Negligible terms in~\eqref{cross-mom-2}:}
We now show that if any $F_{pm_p}$ overlaps with more than one of the $\{F_{tm_t}\}_{t\neq p}$, then $\E X_H$ is a $\smash{O\big(\epsilon\prod_t\var(X_{F_t})^{_{{}^{1/2}}}\big)}$. To do so, we proceed by induction on $k$. In the following, we write that $\{F_{tm_t}\}_{t\in [l]}$ verifies: $(\ast)$ if all $F_{tm_t}$ overlap with at least another $F_{pm_p}$ $t\neq p$; $(\ast\ast)$ if all $F_{tm_t}$ overlap with at exactly one other $F_{pm_p}$ for $p\neq q$. With these definitions, our goal is to prove that
\begin{equation}\label{rec-obj}
\E X_{\cup_t F_{tm_t}}=
\begin{cases}
O\left(\epsilon\prod_t\var(X_{F_t})^{1/2}\right) & \text{if $\{F_{tm_t}\}$ verifies $(\ast)$ but not $(\ast\ast)$,}\\
O\left(\prod_t\var(X_{F_t})^{1/2}\right) & \text{if $\{F_{tm_t}\}$ verifies $(\ast\ast)$.}
\end{cases}
\end{equation}

We can now start our induction. For $k\leq 2$, there is nothing to do. Now fix $k\geq 3$ and assume that for all $l\leq k-1$~\eqref{rec-obj} holds:
\begin{itemize}
\item First, if $\{F_{tm_t}\}_{t\in [k]}$ verifies $(\ast\ast)$. Assume without loss of generality that $F_{1m_1}$ overlaps with $F_{2m_2}$. As vertex disjoint graphs occur independently, $\E X_{\cup_t F_{tm_t}} = \E X_{\cup_{t>2} F_{tm_t}}\E X_{F_{1m_1}\cup F_{2m_2}}$. Then, by \cref{first-order}, and with $J=F_{1m_1}\cap F_{2m_2}$, we get
\[
\E X_{F_{1m_1}\cup F_{2m_2}}=O\big(({\E X_{F_{1m_1}}}/{\sqrt{\E X_J}})({\E F_{2m_2}}/{\sqrt{\E X_J}})\big).
\]
Then, the induction assumption and~\eqref{cross-mom-6}, yield the result.
\item Let us now assume that $\{F_{tm_t}\}_{t\in [k]}$ verifies $(\ast)$ but not $(\ast\ast)$. Then, following~\cite[p7]{rucinski1988small}, there is $p\in[k]$ such that $\{F_{tm_t}\}_{t\in [k]\setminus\{p\}}$ verifies $(\ast)$. Echoing~\citet{rucinski1988small}'s argument, to see this imagine the hypergraph whose vertices are the vertices of $\cup_t F_{tm_t}$ and edges are the vertex-sets $V(F_{tm_t})$, $t\in[k]$. By assumption, there must be at least one connected component with at least three edges, and one of the edges in these components is $V(F_{pm_p})$.

This way we have defined a map from $k$-tuples verifying $(\ast)$ but not $(\ast\ast)$, and $(k-1)$-tuples verifying $(\ast)$. Furthermore, with $J=F_{pm_p} \cap \big(\cup_{t\neq p} F_{tm_t}\big)$, by this mapping every $(k-1)$-tuple is the image of $O(n^{|F_p| - |J|})$ $k$-tuples. Then, using our induction assumption,~\eqref{cross-mom-6}, and recalling that $\E X_J=\Omega(1/\epsilon)$ by construction, we obtain
\begin{align*}
\E X_{\cup_t F_{tm_t}}
&=O\left(n^{|F_p| - |J|}\rho^{e(F_p)-e(J)} \E X_{\cup_{t\neq p} F_{tm_t}}\right)\\
&=O\left((\E X_J)^{-1/2}\frac{\E X_{F_p}}{\sqrt{\E X_J}}\prod_{t\neq p}\var(X_{F_t})^{1/2}\right)\\
&=O\left(\epsilon\prod_t\var(X_{F_t})^{1/2}\right),
\end{align*}
which is the sought after result.
\end{itemize}

\paragraph{Re-factorizing~\eqref{cross-mom-2}:}
Recall that the aim of the above discussion is to determine the leading terms in~\eqref{cross-mom-2}. Equations~\eqref{cross-mom-22} and~\eqref{rec-obj} show that we can ignore $\E\tilde X_H$ in~\eqref{cross-mom-2} if $H=\cup_tF_{tm_t}$ is such that at least one $F_{tm_t}$ overlaps with none or two others. (Recall that $\tilde X_H$ was defined in~\eqref{cross-mom-1-5}.) This yields the result if $k$ is odd, since then the $F_{tm_t}$ cannot be organized in pairs, and we must have $\E\prod_{t=1}^k\big(X_{F_t}-\E X_{F_t}\big) =O\big(\epsilon\prod_{t=1}^k\var(X_{F_t})^{1/2}\big)$.

We now consider the case where $k=2r$ is even. In that case, any term of interest must be built from a partition of $[k]$ into sets of size 2 describing which pair of $F_{tm_t}$ overlap. Let $s=(s_1,\dots,s_r)$ be one such partition, and write $s_t=\{p_t,q_t\}$ for $t\in[r]$. Further, for $t\in[r]$, let $F_{p_t}F_{q_t}$ be the rooted graph obtained by taking the disjoint union of $F_{p_t}$ and $F_{q_t}$ and identifying their roots. With these notations, we have that the only subgraphs $H$ to consider are those obtained by taking the disjoint union of $H_1,\dots, H_r$, where each $H_t$ is in $\smash{\mathcal{H}_{F_{p_t},F_{q_t}}\setminus\{F_{p_t}F_{q_t}\}}$, and identifying their root. In that setting, for one such $H$, using as in~\eqref{cross-mom-22} that vertex disjoint subgraphs occur independently, we have
\[
\E\tilde X_H = \prod_{t=1}^r \E \tilde X_{H_t}.
\]
Therefore, with $\mathcal{H}_{s_t} = \mathcal{H}_{F_{p_t},F_{q_t}}\setminus\{F_{p_t}F_{q_t}\}$ we can write
\[
\E\prod_{t=1}^k\left(X_{F_t}-\E X_{F_t}\right)
=\big(1+O(\epsilon)\big)\!\!
	\sum_{(s_1,\dots,s_r)\in[k]^{(2)}}
	\sum_{H_1\in\mathcal{H}_{s_1}}
	\!\cdots\!
	\sum_{H_r\in\mathcal{H}_{s_r}}
	\prod_{t=1}^r
	\E\tilde X_{H_t}.
\]
There, we recognize the sum of products as the expansion of a product of sums, which factorizes as follows:
\begin{equation}\label{cross-mom-8}
\E\prod_{t=1}^k\left(X_{F_t}-\E X_{F_t}\right)
	=\big(1+O(\epsilon)\big)
		\sum_{(s_1,\dots,s_r)\in[k]^{(2)}}\ 
		\prod_{t=1}^r
		\sum_{H\in\mathcal{H}_{s_t}}
		\E\tilde X_H.
\end{equation}
To proceed, let us consider~\eqref{cross-mom-8} in the special case where $k=2$. Directly replacing in~\eqref{cross-mom-8}, and as $\#[2]^{(2)}=1$, we obtain,
\begin{equation}\label{variance}
\E\left[\left(X_{F_1}-\E X_{F_1}\right)\left(X_{F_2}-\E X_{F_2}\right)\right]
=\big(1+O(\epsilon)\big)\sum_{H\in\mathcal{H}_{\{1,2\}}}
\E\tilde X_H.
\end{equation}
Therefore, we can replace in~\eqref{cross-mom-8} each sums of the form $\smash{\sum_{H\in\mathcal{H}_{s_t}}\E\tilde X_H}$ by the covariance $\cov(X_{F_{p_t}},X_{F_{q_t}}) = \E[(X_{F_{p_t}}-\E X_{F_{p_t}})(X_{F_{q_t}}-\E X_{F_{q_t}})]$, which completes the proof.
\end{proof}

We finally describe the limiting behavior of the rooted subgraph densities introduced in \cref{motif-density-limit}. We reproduce here the statement of the theorem for the reader's convenience.

\thmlocal*

\begin{proof}
We first address the limit in probability before turning to the limit in distribution. In both cases we use~\citet[Theorem 1]{BickelLevina2012} showing that if $n\rho=\Omega(1)$, and with $\hat\rho={e(G)}/{e(K_n)}$, then
\begin{equation}\label{BCL_rho}
\hat\rho =\rho\left(1+O_p\big(n^{-1/2}\big)\right),
\end{equation}
which enables use to replace $\hat\rho$ by $\rho$ (as $n\rho = \Omega(n\rho^{m(F)})$ for any $F$).

\paragraph{Limits in probability:}
By definition, $s_i(F,G)=X_F(G,i)/\hat\rho^{e(F)}X_F(K_n,i)$. Then, if $n\rho^{m(F)}=o(1)$, from \cref{first-order}, $\E X_F(G,i)$ and $\var X_F(G,i)$ are $o(1)$, so that $\P[X_F(G,i)>0]=o(1)$. Therefore, $X_F(G,i)=0$ with high probability in the limit, so that $\P[s_i(F,G)>0]=o(1)$, and we obtain our first claimed result.

If $n\rho^{m(F)}=\omega(1)$, from \cref{first-order}, $X_F(G,i)\to_p \E X_F(G,i)$ and $\E X_F(G,i)= s_x(F,G)\rho^{e(F)}X_F(K_n,i)$. It follows that we have that $X_F(G,i)\to_p s_x(F,G)\rho^{e(F)}X_F(K_n,i)$; i.e., $\tilde s_i(F,G) \to_p s_x(F,\kappa)$. Thus $s_i(F,G) \to_p s_x(F,\kappa)$ from~\eqref{BCL_rho}, which is the result.

\paragraph{Limits in distribution: The univariate case:}
To obtain the limits in distribution we show convergence in moments, and conclude by using that the Normal distribution is uniquely determined by its moment sequence---see for instance~\cite[Theorem~2.a]{kleiber2013moment}. Thus, we consider the $k$-th standardized moment of $s_i(F,G)$. By homogeneity of rescaling and~\eqref{BCL_rho}, we have
\begin{align*}
\E s_i(F,G)^k
&=
\E \left(\frac{X_F-\E X_F}{\sqrt{\var X_F^{\phantom{a}}}}\right)^k\left(1+O_p\big(n^{-1/2}\big)\right).
\end{align*}
Therefore, we can work directly with the standardized moments of $X_F$. Now, from \cref{cross-mom}, setting $F_t=F$ for each $t\in[k]$, we obtain
\begin{equation}\label{thm1-mom-conv}
\E \left(\frac{X_F-\E X_F}{\sqrt{\var X_F^{\phantom{a}}}}\right)^k \to
\begin{cases}
\# [k]^{(2)}&\text{if $k$ is even,}\\
0 &\text{otherwise.}
\end{cases}
\end{equation}
Furthermore, if $k$ is even, $[k]^{(2)}= (k-1)(k-3)\cdots 1=(k-1)!!$, which corresponds to the $k$-th moment of a standardized Normal. Therefore, asymptotically in $n$, $s_i(F,G)$ has the same moments as a standardized Normal distribution, and by~\citep{kleiber2013moment} we obtain the result. The proof in the multivariate case proceeds similarly and is addressed in \cref{proof-motif-density-limit-SI}.
\end{proof}

%%%%%%%%%%%%%%%%%%%%%%%%%%%
\newcommand{\ThmEstim}{\ref{motif-density-clt}}%%%
\section{Proof of Theorem~\protect\ThmEstim{}}%%%
\label{ProofCLT}%%%%%%%%%%%%%%%%%%%
\renewcommand{\theequation}{B.\arabic{equation}}%%
%%%%%%%%%%%%%%%%%%%%%%%%%%%

Up to this point we have only considered rooted densities at a fixed vertex in $G$. We now consider properties of the array $(s_i(F,G))_{i\in[n]}$, and show that the concentration results of \cref{first-order}, as well as some uniform bounds over vertices, permit us to prove \cref{motif-density-clt}.

We first observe that $(s_i(F,G))_{i\in[n]}$ is a row exchangeable triangular array: under our model, for any fixed $n$, $(s_i(F,G))_{i\in[n]}$ is a finitely exchangeable sequence; i.e., for any permutation $\varsigma$ of $[n]$, we have $(s_i(F,G))_{i\in[n]} \sim (s_{\varsigma(i)}(F,G))_{i\in[n]}$. One consequence of exchangeability is that for each $n$ all the $s_i(F,G)$ are identically distributed across $i\in[n]$.

However, exchangeability does not imply independence, and for instance the empirical mean of a row-wise exchangeable triangular array need not be asymptotically Normal~\citep{chernoff1958,blum1958central,weber1980martingale}. Therefore, in the proof that follows, the key ingredient are upper bounds for $\cov\!\big(s_i(F,G), s_j(F,G)\big)$.
\thmglobal*
\begin{proof}
The proof proceeds by approximating $f(s_i(F,G_n), y_i)$ by $f(s_{x_i}(F,\kappa), y_i)$ and $\E f(s_i(F,G_n), y_i)$ by $\E_{(x,y)\sim\mathcal{D}} f(s_x(F,\kappa), y)$, and then showing that the errors are negligible, allowing us to conclude. To simplify notation, we will write $s_i$ for $s_i(F,G_n)$, $s_{x_i}$ for $s_{x_i}(F,\kappa)$, and $x$ for $\{x_i\}_{i\in[n]}$, therefore making the dependence in $n$ implicit. Further, let $K$ be the Lipschitz constant for all $f(\cdot, y)$, $y\in\mathcal{Y}$.

We first approximate $f(s_i, y_i)$ by $f(s_{x_i}, y_i)$. To this end we observe that the $s_i-s_{x_i}$ form an exchangeable array, and therefore are identically distributed. This prompts us to introduce the $\delta_i$-s, independent copies of the $s_i-s_{x_i}$, so that the $\delta_i$-s are i.i.d., and the $s_{x_i}\!\!+\!\delta_i$ are i.i.d. and of the same (marginal) distribution as the $s_i$-s. This allows us to write:
\begin{multline}\label{CLT-proof-0}
\frac1n\sum_{i\in[n]}f(s_i, y_i)-\E f(s_i, y_i)
=\\ \underbrace{\frac1n\sum_{i\in[n]}f(s_{x_i}\!\!+\!\delta_i,y_i)-\E f(s_i, y_i)}_{A_n}
+ \underbrace{\frac1n\sum_{i\in[n]}f(s_i, y_i)-f(s_{x_i}\!\!+\!\delta_i,y_i)}_{B_n}.
\end{multline}

We observe that, under our assumptions, the random variable $A_n$ is the sum of centered independent and identically distributed random variables with finite second moments. Therefore, we can apply the Lindeberg-L\'evy central limit theorem (as $\var s_x >0$ and thus $\Sigma$ is well defined) obtain
\begin{equation}\label{CLT-proof-A}
\sqrt n A_n \xrightarrow[n\rightarrow\infty]{L}
\mathrm{Normal}\left(0,\Sigma\right).
\end{equation}

We now show that $\sqrt n B_n\to_p 0$, which will complete the proof. To do so we consider increasingly complex maps $f$: first we consider $f$ the projection on its first argument, then $f$ Lipschitz but independent of $y_i$, before treating the general case with $s_i$ being multivariate. As we did in the proof of \cref{motif-density-limit}, we note here that $\hat\rho = (e(G_n)/e(K_n)) = (1+O(1/n))\rho$ from~\cite[Theorem 1]{BickelLevina2012}, and in most cases we may replace $\hat\rho$ by $\rho$ without impacting our bounds.

\paragraph{Sum of rooted subgraph counts:}
In this case, we set $f(s_i, y_i) = s_i$, so that $B_n = \tfrac1n\sum_i (s_i-s_{x_i}-\delta_i)$. By construction and \cref{motif-density-limit} (which applies as $\gamma>m$), the $\delta_i$-s are i.i.d., of mean $0$, and second moments tending to $0$. Therefore we can apply the Lindeberg-L\'evy central limit theorem, and obtain that $\tfrac{1}{\sqrt{n}}\sum_i\delta_i \to_p 0$.

Remains therefore to prove that $\tfrac1n\sum_i (s_i-s_{x_i})\to_p 0$. We use the law of total variance to obtain:
\begin{multline}
\label{CLT-proof-B0} 
\var\left(\frac{1}{\sqrt{n}}\sum_i s_i-s_{x_i}\right) = \\
\E \var\left(\frac{1}{\sqrt{n}}\sum_i s_i-s_{x_i} \,\Big\vert\, x\right)
+\var\left(\frac{1}{\sqrt{n}}\sum_i \E\left[s_i-s_{x_i} \,\vert\, x\right]\right)
\end{multline}
We now consider each term of~\eqref{CLT-proof-B0} in succession. 

For the first term in~\eqref{CLT-proof-B0} we use that the $s_{x_i}$ are constant conditionally on $x$ and that the $s_i$ form an exchangeable array to write for any $i\neq j\in[n]$
\begin{equation}
\label{CLT-proof-B1} 
\E \var\left(\frac{1}{\sqrt{n}}\sum_i s_i-s_{x_i} \,\Big\vert\, x\right)
= \E \var\big(s_i\,\vert\, x\big) + n\E \cov\big(s_i, s_j\,\vert\, x\big).
\end{equation}
By \cref{motif-density-limit} (which applies as $\gamma>m$), the first term, the variance, tends to 0. Remains the covariance. There, by the same steps we used to prove first \cref{rooted-copies-product}, and then \cref{first-order}, this covariance is driven by the number of overlapping copies of $F$ rooted at $i$ and $j$ respectively. Here, because of the conditioning, edges form independently, and all terms that do not overlap at least over one edge are independent. Thus, similarly as in~\eqref{cross-mom-6}, and as in~\citet[Eq.~6.3]{BickelLevina2012} and~\citet[pp. 308--309]{nowicki1988}, we have
\[
\E\cov \left(s_i, s_j\,\vert\,x\right) = O\left(
\sum_{H\subset [F,v],\, v\not\in H,\, |e(H)|\geq1}n^{-|H|}\rho^{-e(H)}
\right).
\]
Since $n\rho^\gamma=\omega(1)$, we have that $n^{-|H|}\rho^{-e(H)}=O\big(1/n(n\rho^\gamma)\big)=o(n^{-1})$ for all $H$ in the sum. As the sum is finite, we have
\begin{equation}\label{CLT-proof-B2}
\E \var\left(\frac{1}{\sqrt{n}}\sum_i s_i-s_{x_i} \,\Big\vert\, x\right)=o(1).
\end{equation}

We now consider the second term in~\eqref{CLT-proof-B0}, $\sum_i \E[s_i-s_{x_i} \,\vert\, x]/\sqrt{n}$. As $\E[s_i-s_{x_i} \,\vert\, x] = \E[s_i \,\vert\, x]-s_{x_i}$ (since $s_{x_i} = \E[\,\prod_{pq\in F}\kappa(x_p,x_q) | x_i=x\,]$ is constant conditionally on $x$,) we have to two terms to manipulate: $\sum_i \E[s_i \,\vert\, x]$ and $\sum_i s_{x_i}$. To proceed we now show that $n^{-1}\sum_i \E[s_i \,\vert\, x]$ is a $U$-statistic, and that the $s_{x_i}$ are it's Hajek projections~\citep[Definition~15.2]{DasGupta2008asymptotic}, which will yield the result by~\citep[Theorem~15.1]{DasGupta2008asymptotic}. To see this, we observe that from~\cref{defn:rooted-count} and~\eqref{BCL_rho}, we have
\begin{align*}
s_i 
&= (1+O_p(1/n))\rho^{-e(F)}X_F(G, i)/X_F(K_n, i)\\
&= (1+O_p(1/n))\rho^{-e(F)}\aut(F)\left((n-1)_{|F|-1}\right)^{-1}\sum_{H\subset K_n\,:\, [H,i]\equiv F}\II\{H\subset G\}.
\end{align*}
As for any $H$ in the sum, $\E[\II\{H\subset G\} \,\vert\, x] = \rho^{e(F)}\prod_{pq\in H}\kappa(x_p,x_q)$, summing over $i$, we obtain:
\begin{align*}
n^{-1}\sum_i \E[s_i \,\vert\, x]
&= n^{-1}(1+O_p(1/n))\big((n-1)_{|F|-1}\big)^{-1}\sum_{i_1, \ldots, i_{|F|}\in [n]} \,\prod_{pq\in F}\kappa(x_{i_p}, x_{i_q})\\
&= (1+O_p(1/n))\binom{n}{|F|}^{-1}\sum_{i_1< \ldots< i_{|F|}\in [n]} \,\prod_{pq\in F}\kappa(x_{i_p}, x_{i_q}),
\end{align*}
where, up to a negligible factor, we recognize a $U$-statistic of order $|F|$ and of kernel $\prod_{pq\in F}\kappa(x_p, x_q)$. It then directly follows that $s_{x_i}$ are its Hajek projections, as $s_{x_i} = \E[\,\prod_{pq\in F}\kappa(x_p,x_q) | x_i=x\,]$, which by~\citep[Theorem~15.1]{DasGupta2008asymptotic}, (which applies as $s_x$ is twice integrable because it is bounded, and $\var s_x >0$ and the kernel is non-degenerate) yields that
\begin{equation}\label{CLT-proof-B3}
\sqrt{n}\left(\frac{1}{n}\sum_i \E[s_i \,\vert\, x] - \frac{1}{n}\sum_i s_{x_i}\right)\to 0.
\end{equation}

Together,~\eqref{CLT-proof-B0},~\eqref{CLT-proof-B2}, and~\eqref{CLT-proof-B3} yield that 
\begin{equation}\label{CLT-proof-C}
\sqrt{n}B_n = o_p(1),
\end{equation}
which is the sought-after result.

\paragraph{The univariate case:}
Here we consider the case where $f(s_i,X_i)$ does not depend on $Y_i$ but is $K$-Lipschitz in its first argument. Then, by the triangular inequality,
\begin{equation}\label{CLT-proof-B4}
\left|\sqrt{n}B_n\right|
\leq\frac{1}{\sqrt{n}}\sum_{i\in[n]}\left|f(s_i)-f(s_{x_i}+\delta_i)\right|
\leq K \frac{1}{\sqrt{n}}\sum_{i\in[n]}\left|s_i-s_{x_i}-\delta_i\right|.
\end{equation}
There, as $\E B_n = 0$ by construction, and writing $\Delta_i = s_i-s_{x_i}-\delta_i$, paralleling the previous segment of this proof, we have:
\begin{equation*}
\left|\sqrt{n}B_n\right|
= O_p\left(\sqrt{\var |\Delta_i|+n\cov(|\Delta_i|,|\Delta_j|)}\right).
\end{equation*}
Now, in the previous section we have shown that $\sum_i\Delta_i/\sqrt{n} \to_p 0$, and therefore we also have proved that $\var \Delta_i+n\cov(\Delta_i, \Delta_j) \to 0$. As $\var|\Delta_i|=O(\var\Delta_i )$ and $\cov(|\Delta_i|,|\Delta_j|) = O\big(\cov(\Delta_i,\Delta_j)\big) $ we obtain that
\begin{equation}\label{CLT-proof-B5}
\left|\sqrt{n}B_n\right|
\leq K \frac{1}{\sqrt{n}}\sum_{i\in[n]}\left|s_i-s_{x_i}-\delta_i\right| = o_p(1),
\end{equation}
which is the sought after result. For the covariance bound, while the result is intuitive in the sense that covariations are constrained by the absolute value, it can also be seen as an application of~\citet[Theorem~1]{cuadras2002}, where it is shown that $f: \mathbb{R}\to\mathbb{R}$ with a bounded derivative and for two real valued random variables $U$, $V$ such that $\big(f(U), f(V)\big)$ has a bounded second moment, we have that $\cov(f(U), f(V)) < O\big(\cov(U,V)\big)$.

\paragraph{The general case:}
In this case we have $F=(F_1,\dots,F_k)$ for $k\geq 1$, and we naturally expand $s_i$, $s_{x_i}$ to being $k$-variate vectors. Using the subadditivity of the square root, we observe that
\begin{align*}
\|\sqrt{n}B_n\|_2
&\leq K \frac{1}{\sqrt{n}}\sum_{i\in[n]} \|s_i-s_{x_i}-\delta_i\|_2\\
&\leq K\sum_{t\in[k]}\frac{1}{\sqrt{n}}\sum_{i\in[n]}\left|s_i(F_t,G)-s_{x_i}(F_t,\kappa)-\delta_{it}\right|.
\end{align*}
Then,~\eqref{CLT-proof-B5} yields that $\sqrt n B_n =o_p(1)$ in that case as well.
\end{proof}
\begin{Remark}
If $\var s_x = 0$, we have $\sqrt{n}A_n \to_p 0$ and the leading term is $B_n$. Therefore, we do not converge to a Normal, but rather towards a mass at 0. However,~\citep[Theorem~15.6]{DasGupta2008asymptotic} applies on the $U$-statistic component of $B_n$, and we conjecture that $nB_n$ tends in distribution to a mixture of $\chi_2$ distributions. \Citet[Theorem 1.1.c]{hladky2019} shows that this conjecture holds at least for cliques. Proving this conjecture would require tighter bounds in~\eqref{CLT-proof-B2} specific to this regime.
\end{Remark}

%%%%%%%%%%%%%%%%%%%%%%%%%%%
\newcommand{\methodo}{\ref{Method}}%%%
\section{Material for Section \protect\methodo}%%%%
\label{metho}%%%%%%%%%%%%%%%%%%%
\renewcommand{\theequation}{C.\arabic{equation}}%%
%%%%%%%%%%%%%%%%%%%%%%%%%%%
\subsection{Estimating $\kappa$}\label{kappahat}
We use maximum likelihood to estimate $\kappa$. As doing so requires to perform an optimization over the space of measure preserving transformations of $[0,1]$, we considered several approaches:~\cite{olhede2013network,lei2015,chatterjee2015,sarkar2015}, as well as \cref{alg:kappa}. Our experiments show that \cref{alg:kappa} performed best on synthetic data, while~\cite{olhede2013network} performed best on the dataset we studied. Thus, in \cref{Method}, we only use~\cite{olhede2013network}\footnote{Using~\citet{olhede2013network}'s code: {\tt github.com/p-wolfe/network-histogram-code}.}. Furthermore, in using~\cite{olhede2013network}, we made the following augmentations: i) the approach relies on a rule of thumb to determine the number of blocks $\hat b$ to use, we considered all number of blocks in $[.9\hat b, 1.1\hat b]$; ii) for each considered number of blocks, as the algorithm is randomized, we perform the estimation $10$ times and kept the estimate with the maximal likelihood; iii) when comparing likelihoods from estimates with different number of blocks, as the model is then parametric, we use information criterions to correct for the difference in number of parameters (here AIC).

%%%%%%%%%%%%%%%%%%%%%%%%%%%%%%%%%%%%%
\setlength{\textfloatsep}{.9\baselineskip}%%%%%%%%%%%%%%%%%
\begin{algorithm}%%%%%%%%%%%%%%%%%%%%%%%%%%%%
	\SetAlgoLined
	\KwInput{A graph $G$}
	\KwOutput{A kernel estimate $\hat\kappa$ and block assignments $\{P_b\}_{b\in[k]}$}
	{Use Louvain algorithm to organize $G$'s vertices in blocks. Call $k$ the ensuing number of blocks. For each $b\in[k]$, let $P_b$ be the set of vertices in block labeled $b$\;}
	{Set $\hat B\in [0,1]^{k\times k}$, such that $\forall(b,b')\in[k]^2,\ \hat B_{bb'} ={\#\{ij\in G\,:\, i\in P_{b\vphantom{b'}}, j\in P_{b'}\}}/{|P_{b\vphantom{b'}}|(|P_{b'}|-\II{\{b=b'\}})}$\;}
	{Set $\hat\sigma: (0,1) \to [k]$, such that $\forall u\in(0,1),\ \hat\sigma(u) = \sum_{b\in[k]}\II{\big\{u>\sum_{b'\leq b} |P_{b'}|/n\big\}}$\;}
	{Set $\hat\kappa: (0,1)^2\to [0,1]$, such that $\forall u,v\in(0,1)^2,\ \hat\kappa(u,v) = \hat B_{\hat\sigma(u),\hat\sigma(v)}$\;
	}
\caption{\label{alg:kappa}Block-kernel estimator.}
\end{algorithm}%%%%%%%%%%%%%%%%%%%%%%%%%%%%%
%%%%%%%%%%%%%%%%%%%%%%%%%%%%%%%%%%%%%

\subsection{Producing the $\hat t_i$-s}
From \cref{motif-density-limit} we know that if $(G_n)_{n>0}\sim G(\rho,\kappa)$, and vertex $i$ has latent position $x$, then there is a sequence $\Sigma_x(n,\kappa)$ such that
\begin{equation*}
t_i(G_n,x,\kappa)
	= \Sigma_x(n,\kappa)^{-1/2}\big(
			s_i(\TexttriangleR\ ,G_n)-s_x(\TexttriangleR\ ,\kappa),
			s_i(\TextsquareR\ ,G_n)-s_x(\TextsquareR\ ,\kappa)
		\big)
\end{equation*}
is asymptotically bivariate standard normal. We estimate $s_x(\TexttriangleR\ ,\kappa)$, $s_x(\TextsquareR\ ,\kappa)$, and $\Sigma_x(n,\kappa)$ using bootstrap. In doing so, we observe that since $\hat\kappa$ is a step-function, the densities are step-function themselves, and we may average our estimates across blocks, as well as index them by block label rather than $x$. See \cref{alg:covs}. We call $\hat t_i(G)$ the estimator of $t_i$ resulting from using \cref{kappahat} and \cref{alg:covs}. We observe through simulation that the $\hat t_i$-s are already close to normal for $n=5000$ (see \cref{fig:QQ}.)

\begin{Remark}
The parametric bootstrap is computationally expensive, as it requires to count subgraphs in each replicate. However, as $\Sigma_x(n,\rho\kappa)$ depends on the decay rate of $\rho$, computing $\Sigma_x(n,\rho\kappa)$ directly would require assuming one such rate. The parametric bootstrap avoids this assumption. However, averaging across vertices in each block, as in \cref{alg:covs}, improves precision by a factor of order $1/\sqrt{n}$, and only few replicates are sufficient. An alternative is to evaluate the moments directly though $\mathcal{H}_F$ and associated $s_x(H,\kappa)$, as in \cref{var-closed-form} and~\cite{maugis2017network}. In that case, given $\mathcal{H}_F$, the computational complexity will be that of evaluating $k^{2|F|-1}$ integrals over $[0,1]^{2|F|}$.
\end{Remark}

%%%%%%%%%%%%%%%%%%%%%%%%%%%%%%%%%%%%%
\setlength{\textfloatsep}{.9\baselineskip}%%%%%%%%%%%%%%%%%
\begin{algorithm}%%%%%%%%%%%%%%%%%%%%%%%%%%%%
	\SetAlgoLined
	\KwInput{A blockmodel $\kappa$, a network size $n$, and a number of replicates $N$}
	\KwOutput{$\hat s_b(\TexttriangleR\ ,\kappa)$, $\hat s_b(\TextsquareR\ ,\kappa)$ and $\widehat\Sigma_b(n,\kappa)$ for $b\in[k]$}
	{Set $k$ to be the number of blocks in $\kappa$\;}
	{For $i\in[n]$, set $x_i = i/(n+1)$, and for $b\in[k]$ let $P_b$ be the set of vertices in block $b$\;}
	{Generate $G_1, \dots, G_N$, $N$ independent realizations of $G(n,\kappa)\vert \{x_i\}_{i\in[n]}$\;}
	{Compute the $s_i(\TexttriangleR\ ,G_t)$ and $s_i(\TextsquareR\ ,G_t)$ for $i\in[n]$ and $t\in[N]$\;}
	{Compute for each $b\in[k]$ $\begin{cases}
		\hat s_b(\TexttriangleR\ ,\hat\kappa) &= \frac{1}{N}\sum_{t=1}^N
			\frac{1}{|P_b|}\sum_{i\in P_b} s_i(\TexttriangleR\ ,G_t),\\
		\hat s_b(\TextsquareR\ ,\hat\kappa) &= \frac{1}{N}\sum_{t=1}^N
			\frac{1}{|P_b|}\sum_{i\in P_b} s_i(\TextsquareR\ ,G_t),\\
		\widehat \Sigma_b(n,\hat\kappa) &= \frac{1}{N-1}\sum_{t=1}^N
			\frac{1}{|P_b|}\sum_{i\in P_b}
			\big(s_i(\TexttriangleR\ ,G_t)-\hat s_b(\TexttriangleR\ ,\hat\kappa)\big)
			\big(s_i(\TextsquareR\ ,G_t)-\hat s_b(\TextsquareR\ ,\hat\kappa)\big);
	\end{cases}$}
\caption{\label{alg:covs} Estimator of rooted count distribution parameter.}
\end{algorithm}%%%%%%%%%%%%%%%%%%%%%%%%%%%%%
%%%%%%%%%%%%%%%%%%%%%%%%%%%%%%%%%%%%%

\subsection{Critical value}\label{cval}
Our proofs imply that the convergence of \cref{motif-density-limit} is not uniform: higher order moments, while converging at the same rate, exhibit exponentially larger constant factors. Thus, extremal quantiles cannot be expected to have converged when performing analysis, which we confirm through simulation in \cref{fig:QQ}. We note here that~\cite{sileikis2019} suggests that convergence may become uniform for dense enough random graph sequences $G$.

We therefore adapt the critical value of the test, which we do through bootstrap. Specifically, while for a given level $\alpha$ and using Bonferroni correction the asymptotic critical value is $\smash{F^{-1}_{{}^{\chi^2(2)}}}(1-\alpha/n)$, we instead use $(1-\alpha)$-th quantile of $\max_{i\in [n], r\in[R]} \{\hat t_i(G_r)\}$ (where the $\{G_r\}_{r\in[R]}$ are random graph replicates.) We present in \cref{fig:QQ} an example of how the critical value is augmented. Finally, using simulation we confirm that doing so preserves both the level and power of the test.

%%%%%%%%%%%%%%%%%%%%%%%%%%%%%%%%%%%%%
\begin{figure}[t]%%%%%%%%%%%%%%%%%%%%%%%%%%%%%%
\centering
\includegraphics[width=.325\textwidth]{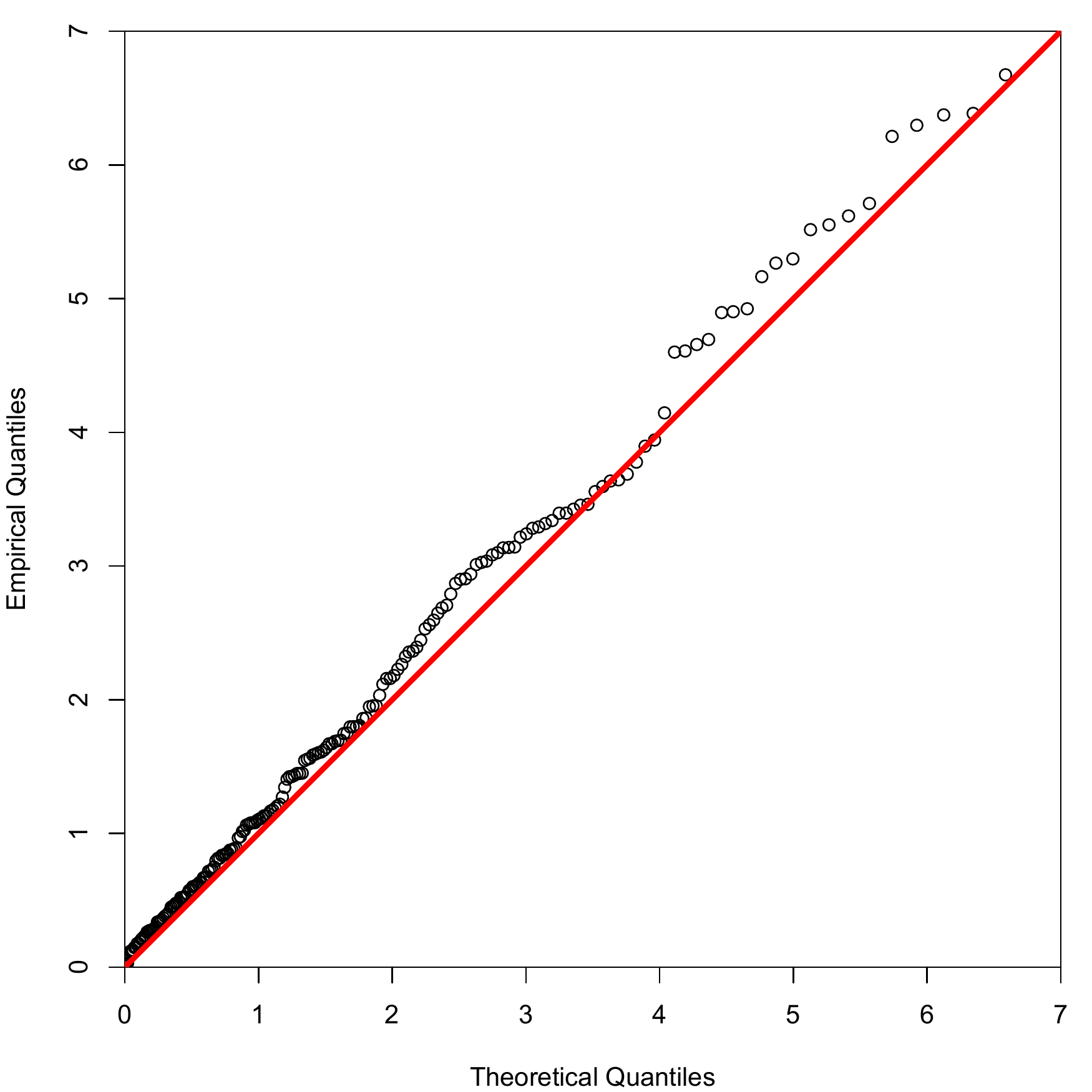}
\includegraphics[width=.325\textwidth]{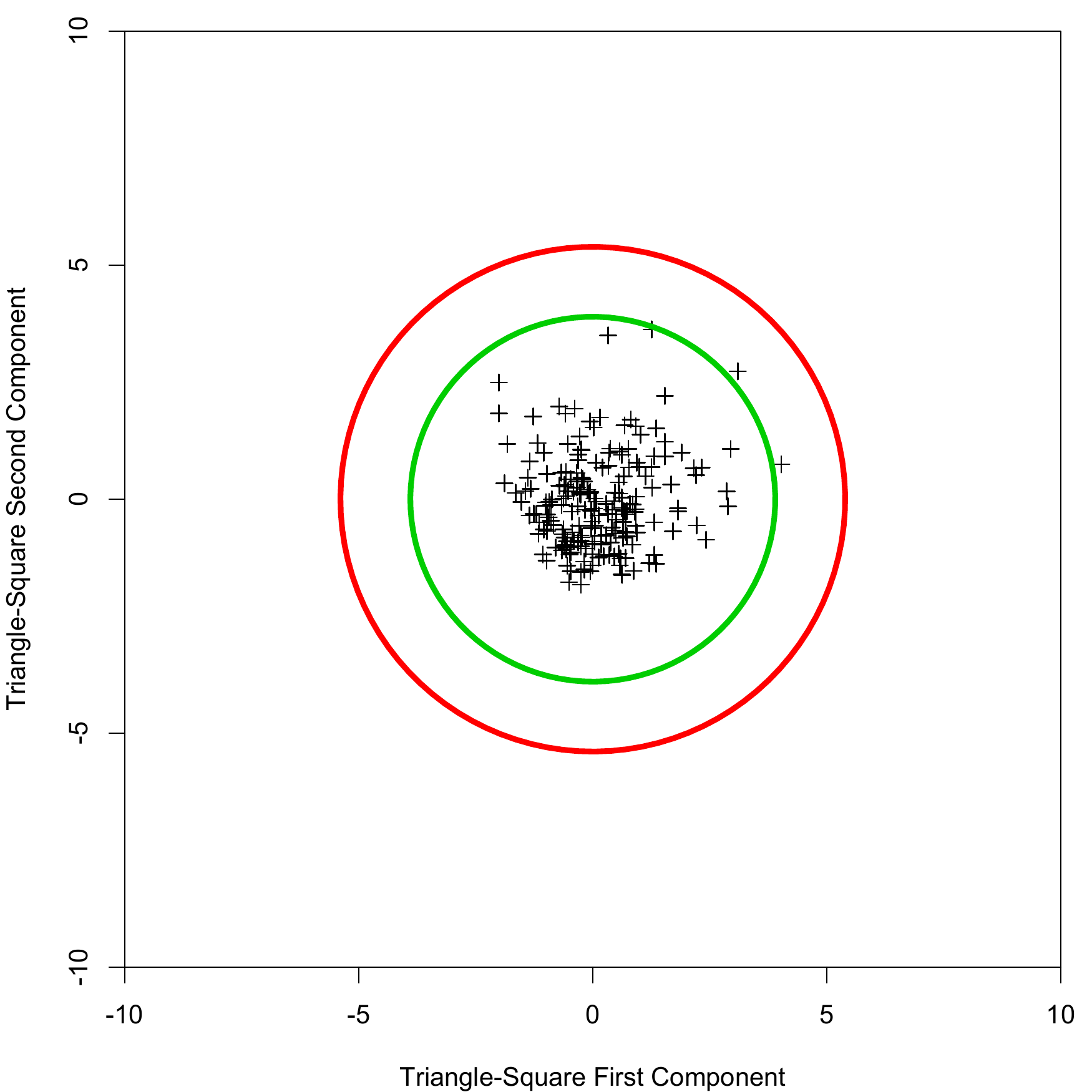}
\includegraphics[width=.325\textwidth]{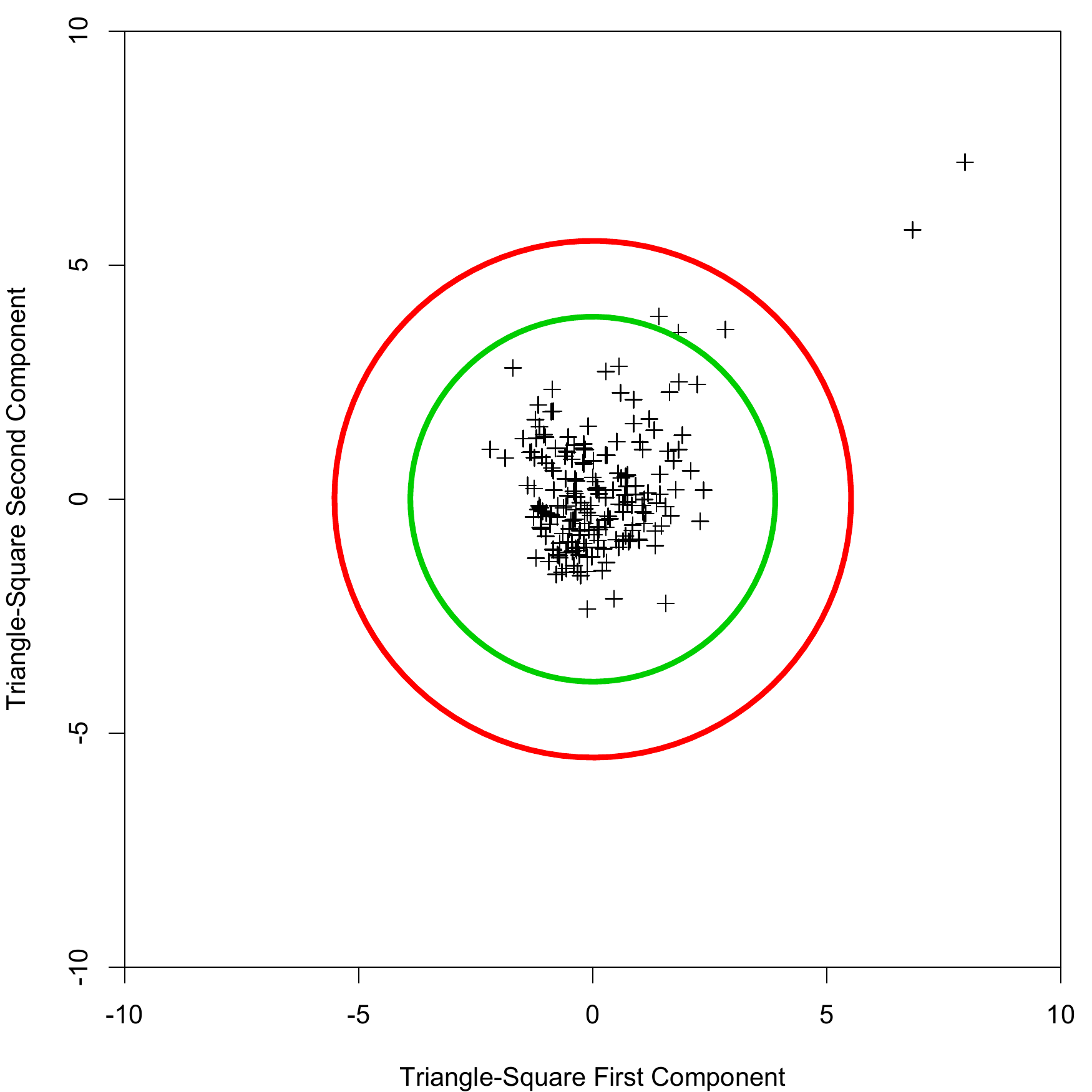}
\caption{Simulation experiments. %
\textit{\textbf{Left:}} QQ-Plot of realized distribution of standardized rooted counts against theoretical distribution. We simulate $\{G_l\}_{l\in[200]}$, each of order $n=5000$ from $\kappa$ a block model with $3$ blocks of equal sizes and sparsity $\rho=\smash{n^{{}_{-1/3}}}$. From $\{G_l\}_{l\in[200]}$ we produce the $\smash{\{t_1(G_l)\}_{l\in[200]}}$. We present the QQ-plot of $\smash{\{\|\hat t_1(G_l)\|_2^2\}_l}$ against the $\smash{\chi^{_2}}(2)$ distribution. We do not observe significant departure from the $\chi^{_{2}}(2)$ distribution.
\textit{\textbf{Center:}} Exemplar movement of critical value. We use the same set-up as in the leftmost plot, but with $n=200$. Crosses are the $\{\hat t_i(G_1)\}_{i\in[200]}$. The green circle is the asymptotic critical value ($\smash{F^{_{-1}}_{{}^{\chi^{_{2}}}}(2)}(1-\alpha/n)$), in red is the estimated one. With the estimated critical value, we recover the selected level of the test; i.e., we perform the test on $G_1$ to $G_{200}$, and we cannot reject the null that the rejection probability is $\alpha$.
\textit{\textbf{Right:}} Triadic closure and test power. We use the same set-up as in the central plot. Then, we enforce $5\%$ triadic closure on $G_1$ (we closed $5\%$ of the opened triangles attached to $5\%$ of the vertices.) The plot is structured as the central one, and we observe two points outside the confidence zone. The observed rejection rate across graph replicates is not significantly different from $1$.}
\label{fig:QQ}
\end{figure}%%%%%%%%%%%%%%%%%%%%%%%%%%%%%%%
%%%%%%%%%%%%%%%%%%%%%%%%%%%%%%%%%%%%%

%%%%%%%%%%%%%%%%%%%%%%%%%%%%%%%%%%%%%
\section{Proof of~\cref{first-order}}%%%%%%%%%%%%%%%%%%%%
\label{proof-of-first-order}
%%%%%%%%%%%%%%%%%%%%%%%%%%%%%%%%%%%%%
In what follows, we deconstruct~\cref{first-order} in three results that we prove successively. This allows us to put forward one key intermediate result: an homomorphism property for rooted counts. This property mirrors results that have proved central to the study of subgraph counts at the global scale (see~\citet[Equation~2,~p964]{lovasz2006rank},~\citet[Equation~1,~p217]{bollobas2009metric} and~\citet[Equation~7.7,~p117]{lovasz2012large}.)

\begin{Lemma}\phantomsection\label{nb-copy-order}
Fix a rooted graph $F$ and let $G\sim G(n,\rho\kappa)|x_i=x$. Then
\begin{align}
\label{nb-copy-order-1}
\E X_F&=s_x(F,\kappa)\,\rho^{e(F)}X_F(K_n,i),\\
\label{nb-copy-order-2}
X_F&=O_p\left(n^{|F|-1}\rho^{e(F)}\right).
\end{align}
where $s_x(F,\kappa) = \E \big[\prod_{pq\in F}\kappa(x_p,x_q)\,\vert\, x_i=x\big]$.
\end{Lemma}

The $s_x(F,\kappa)$ notation in consistent with the literature: it uses the `$s$' of~\cite{bollobas2009metric} to denote an embedding density, and the `$x$' subscript of~\cite{lovasz2014automorphism} to denote the density of a partially labelled graph.

\begin{proof}[Proof of~\cref{nb-copy-order}]
We prove both statements in succession. We start by evaluating the mean. To do so, we write the number of copies of $F$ rooted at $i$ in $G$ as follows:
\begin{equation}\label{nb-of-copies}
X_F = \#\{F'\subset G : i\in F', [F',i]\equiv F\} = \sum_{
\substack{i\in F'\subset K_n\\ [F',i]\equiv F}}\II{\{F'\subset G\}}.
\end{equation}
We now exploit the conditional independence of edge occurrences, through the law of iterated expectation, to compute $\smash{\E\II{\{F'\subset G\}}}$. We obtain that for any $[F',i]\equiv F$,
\begin{align}
\nonumber
\E\II{\{F'\subset G\}}
&= \E \prod_{pq\in F'}\II\{pq\in G\}\\
\nonumber
&= \E\left[\E \left[\,\prod_{pq\in F'}\II\{pq\in G\}\,\Big\vert\, \{x_t\}_{t\in F'\setminus i}\,\right]\right]\\
&= \E \left[\,\prod_{pq\in F'}\rho\kappa(x_p,x_q)\,\Big\vert\, x_i=x\,\right]
\label{nb-of-copies-2}
 = s_x(F,\kappa)\rho^{e(F)}.
\end{align}
In this last expression, we observe that $\E\smash{\II{\{F'\subset G\}}}$ does not depend on the chosen copy $F'$, but is a function of $x$, $F$, $\kappa$ and $\rho$. Therefore, from~\eqref{nb-of-copies} and~\eqref{nb-of-copies-2}, using the linearity of the expectation operator we obtain that 
\[
\E X_F = s_x(F,\kappa) \rho^{e(F)}X_F(K_n,i),
\]
which recovers~\eqref{nb-copy-order-1}, and therefore the first half of the sought-after result.

We now consider the bound in probability of~\eqref{nb-copy-order-2}. First, denoting $\aut(F)$ the cardinality of the root preserving automorphism group of $F$ as in~\cref{defn:rooted-count}, we have
\[
X_F(K_n,i) 
= \frac{(n)_{|F|-1}}{\aut(F)}=\Theta\left(n^{|F|-1}\right),
\]
as $\aut(F)$ is fixed and greater than $0$. Then, as $s_x(F,\kappa)>0$ by Definition~1, $\E X_F = \Theta\big(n^{|F|-1}\rho^{e(F)}\big)$, and we may fix $M>0$ and $n_0$ such that
\[
\forall n\geq n_0,\quad\E X_F \leq Mn^{|F|-1}\rho^{e(F)}.
\]
Finally, since $0\leq X_F$ for all $n$, we can use Markov's inequality and obtain that for any $\epsilon<1$,
\begin{equation*}
\forall n\geq n_0,\quad
\P\left[\frac{X_F}{n^{|F|-1}\rho^{e(F)}} \geq \epsilon^{-1}M\right] \leq \frac{\E X_F(G,i)/n^{|F|-1}\rho^{e(F)}}{\epsilon^{-1}M}\leq\epsilon,
\end{equation*}
which is equivalent to $X_F = O_p\big(n^{|F|-1}\rho^{e(F)}\big)$, the sought-after result.
\end{proof}

The key consequence of~\cref{nb-copy-order} is that the order of magnitude of $X_F$ is driven solely by the number of vertices and edges in $F$, paralleling results presented~\cite{erdos59:_on_random_graphs} and~\cite{rucinski1988small} for the Erd\H{o}s-R\'enyi model. This translates into set operations on rooted graphs presenting good properties in terms of order or magnitude of number of rooted copies. We present two such properties in the following propositions.

\begin{Proposition}\phantomsection\label{motif-set-operations-p}
Set $F_1,F_2$ to be two rooted graphs sharing their root and let $G\sim G(n,\rho\kappa)|x_i=x$. Then, $X_{F_1\cup F_2}
= O_p\big(\E X_{F_1}\E X_{F_2}/\E X_{F_1\cap F_2}\big)$, and if $F_1\subset F_2$, $\P[X_{F_2}>0] = O\big(\E X_{F_1}\big)$.
\end{Proposition}

\begin{proof}
We prove both results in succession. Let us begin with the first assertion, and write using~\cref{nb-copy-order} and the classical equality $|F_1\cup F_2| = |F_1|+|F_2|-|F_1\cap F_2|$,
\begin{align*}
X_{F_1\cup F_2}
&= O_p\left(n^{(|F_1| + |F_2| - |F_1\cap F_2|)-1}\rho^{e(F_1)+e(F_2)-e(F_1\cap F_2)}\right)\\
&= O_p\left(\frac{n^{|F_1|-1}\rho^{e(F_1)}n^{|F_2|-1}\rho^{e(F_2)}}{n^{|F_1\cap F_2|-1}\rho^{e(F_1\cap F_2)}}\right)\\
&= O_p\left(\frac{\E X_{F_1}\E X_{F_2}}{\E X_{F_1\cap F_2}}\right),
\end{align*}
which is the sought-after result.

We now turn to the second assertion and first write that $\P[X_{F_2}>0]=\P[X_{F_2}\geq 1]$. Then, since $F_1\subset F_2$, if there is a copy of $F_2$, there must also be one of $F_1$, so that $\P[X_{F_2}\geq1]\leq\P[X_{F_1}\geq 1]$. Then, using Markov's inequality, we obtain $\P[X_{F_2}>0]
	\leq \P[X_{F_1}\geq 1]
	\leq \E X_{F_1}$,
which is sufficient to obtain the sought-after result.
\end{proof}

\cref{motif-set-operations-p} shows that if any rooted subgraph of $F$ has a vanishing expected number of rooted copies, then $X_F=0$ with high probability (generalizing~\citet[Lemma~1]{rucinski1986balanced}). Conversely,~\cref{motif-set-operations-p} also shows that if all subgraphs of both $F_1$ and $F_2$ have diverging expected number of rooted copies, then $X_{F_1\cup F_2} = o_p(\E X_{F_1}\!\E X_{F_2})$ when $|F_1\cap F_2|>1$.

To characterize when all subgraphs have diverging expected number of rooted copies, we use the parameter $m(F)$:
\[
m(F) = \max\left\{{e(H)}/{|H|-1}\,:\,v\in H\subset F, |H|>1\right\}.
\]
Now observe that for each $v\in H\subset F$, $n^{|H|-1}\rho^{e(H)}=\omega(1)$ if and only if $n\rho^{e(H)/(|H|-1)}=\omega(1)$. Furthermore, as $\rho\leq 1$, we have that $n\rho^{e(H)/(|H|-1)}\geq n\rho^{m(F)}$. Therefore, similarly as in~\citet[Lemma~1]{rucinski1986balanced}, we observe a type of zero-one law: if $n\rho^{m(F)}=o(1)$, then $X_F$ is almost surely equal to zero, while if $n\rho^{\max(m(F_1),m(F_2))}=\omega(1)$, then the number of rooted copies of elements of $\mathcal{H}_{F_1,F_2}$ is dominated by $\E X_{F_1}\!\E X_{F_2}$.

These results, in conjunction with~\cref{rooted-copies-product}, will now allow us to prove that rooted graph counts verify a partial homomorphism property (the product of rooted subgraph counts approximately maps to the rooted subgraph count of their gluing product; see definition just below and~\cref{c-H-examples} for examples.) Ultimately, this implies the concentration of each count to its expectation for $\rho$ large enough. This result generalizes~\citet[Equation~2,~p964]{lovasz2006rank},~\citet[Equation~1,~p217]{bollobas2009metric} and~\citet[Equation~7.7,~p117]{lovasz2012large}.

\newcommand{\CiteLoc}{\cite[Section~4.2]{lovasz2012large}}%%%
\begin{Definition}[Gluing Product~\protect{\CiteLoc{}}]\phantomsection\label{gluing-product}%%%
We call gluing product of two rooted graphs $F_1,F_2$---and write $F_1F_2$---the rooted graph obtained by taking the disjoint union of $F_1$ and $F_2$, and then merging the two root vertices. We write $F_1^k$ for the $k$-fold gluing product of $F_1$ with itself.
\end{Definition}

\begin{Proposition}\phantomsection\label{hom-prop}%
Fix two rooted graphs $F_1, F_2$. Set $G\!\sim\!G(n,\rho\kappa)|x_i\!=x$ and $\epsilon_F\!=\!1/\min\!\big\{n^{|H|-1}\rho^{e(H)}\!: v\!\in\!H\!\subset\!F\big\}$. Then, $\epsilon_F=O\big(1/n\rho^{m(F)}\big)$ and
\[\begin{cases}
X_{F_1}X_{F_2} = c_{F_1F_2}X_{F_1F_2}
	\left(1+O_p(\epsilon_{F_1F_2})\right)
	&\text{if $n\rho^{m(F_1F_2)}=\omega(1)$,}\\
X_{F_1}/\E X_{F_1} = 
	1+O_p(\epsilon_{F_1})
	&\text{if $n\rho^{m(F_1)}=\omega(1)$.}
\end{cases}\]
\end{Proposition}

\begin{proof}
We expand $X_{F_1}X_{F_2}$ and then obtain the two claimed statements by identifying the leading terms in the obtained expression. We start from~\cref{rooted-copies-product}, that we recall for the reader's convenience,
\begin{equation*}
X_{F_1}X_{F_2} = \sum_{H\in\mathcal{H}_{F_1,F_2}} c_HX_H.
\end{equation*}
Then, we use~\cref{gluing-product} to split the sum in two terms as follows:
\begin{equation}\label{hom-prop-eq1}
X_{F_1}X_{F_2}
 = c_{F_1F_2}X_{F_1F_2}
 +\sum_{H\in\mathcal{H}_{F_1,F_2}\setminus\{F_1F_2\}} c_HX_H.
\end{equation}

We now use~\eqref{hom-prop-eq1} to prove that $X_{F_1}X_{F_2} = c_{F_1F_2}X_{F_1F_2}\big(1+O_p(\epsilon_{F_1F_2})\big)$. To control the sum on the right, we first fix $H\in\mathcal{H}_{F_1,F_2}\setminus\{F_1F_2\}$ and two rooted graphs $F_1'$ and $F_2'$ such that $H = F_1'\cup F_2'$. Then, we note that by construction $\max(m(F_1),m(F_2)) = m(F_1F_2)$. Therefore, the assumption that $n\rho^{m(F_1F_2)}=\omega(1)$ along with~\cref{motif-set-operations-p} directly implies that
\begin{equation}\label{hom-prop-eq2}
X_H
= O_p\left(\frac{\E X_{F_1'}\!\E X_{F_2'}}{\E X_{F_1'\cap F_2'}}\right)
= O_p\left(\epsilon_{F_1F_2}\E X_{F_1}\!\E X_{F_2}\right).
\end{equation}
Now, recall that $X_F(K_n) = (n)_{|F|-1}/\aut(F)$, with $(n)_k = n(n-1)\cdots(n-k+1)$ the falling factorial, so that from~\cref{nb-copy-order} we obtain
\begin{align*}
\E X_{F_1}\E X_{F_2}
&=\left(\frac{\left(n\right)_{|F_1|-1}}{\aut(F_1)}
		\rho^{e(F_1)}s_x(F_1,\kappa)\right)
	\left(\frac{\left(n\right)_{|F_2|-1}}{\aut(F_2)}
		\rho^{e(F_2)}s_x(F_2,\kappa)\right)\\
&=\frac{n^{|F_1|+|F_2|-2}}{\aut(F_1)\aut(F_2)}
		\rho^{e(F_1)+e(F_2)}\left(1+O\left(n^{-1}\right)\right)\\
&\qquad\quad
	\E\left[\,
	\prod_{pq\in F_1}\kappa(x_p,x_q)
	\,\Big\vert\, x_i=x\,\right]
	\E\left[\,
	\prod_{pq\in F_2}\kappa(x_p,x_q)
	\,\Big\vert\, x_i=x\,\right].
\end{align*}
We now use the definition of the gluing product---especially that by construction $|F_1F_2|=|F_1|+|F_2|-1$ and $e(F_1F_2)=e(F_1)+e(F_2)$---to factorize the two expectations, and obtain
\begin{align}
\nonumber
\E X_{F_1}\!\E X_{F_2}\!
&= \frac{n^{|F_1F_2|-1}}{\aut(F_1)\aut(F_2)}
		\rho^{e(F_1F_2)}\left(1+O\left(n^{-1}\right)\right)\\
\nonumber
&\qquad\qquad\qquad
	\E\left[\,
	\prod_{pq\in F_1F_2}\kappa(x_p,x_q)
	\,\Big\vert\, x_i=x
	\,\right]\\
\nonumber
&= \frac{\aut(F_1F_2)}{\aut(F_1)\aut(F_2)}
		\frac{(n)_{|F_1F_2|-1}}{\aut(F_1F_2)}\rho^{e(F_1F_2)}
		s_x(F_1F_2,\kappa)\left(1+O\left(n^{-1}\right)\right)\\
\nonumber
&= \frac{\aut(F_1F_2)}{\aut(F_1)\aut(F_2)}
		s_x(F_1F_2,\kappa)\rho^{e(F_1F_2)}X_F(K_n)
		\left(1+O\left(n^{-1}\right)\right)\\
\label{hom-prop-eq3}
&=c_{F_1F_2}\E X_{F_1F_2}\left(1+O\left(n^{-1}\right)\right).
\end{align}
Together~\eqref{hom-prop-eq1},~\eqref{hom-prop-eq2} and~\eqref{hom-prop-eq3} imply that
\begin{equation}\label{hom-prop-eq4}
X_{F_1}X_{F_2} = c_{F_1F_2}X_{F_1F_2}\big(1+O_p(\epsilon_{F_1F_2})\big),
\end{equation}
which is the sought-after result.

We now use~\eqref{hom-prop-eq4} to prove that $X_{F_1}=\E X_{F_1}\big(1+O_p(\epsilon_{F_1})\big)$. We do so using the second moment method. We begin by replacing $F_2$ by $F_1$ in both~\eqref{hom-prop-eq3} and~\eqref{hom-prop-eq4}, to obtain
\begin{multline*}\begin{dcases}
\E X_{F_1}\E X_{F_1}
	= c_{F_1F_1}\E X_{F_1F_1}\big(1+O(n^{-1})\big)
	\vphantom{\left(\E X_{F_1}\right)^2}\\
X_{F_1}X_{F_1} = c_{F_1F_1}X_{F_1F_1}\big(1+O_p(\epsilon_{F_1})\big)
	\vphantom{o_p\left(\E X_{F_1^2}\right)}
\end{dcases}
\Rightarrow\\
\begin{dcases}
\left(\E X_{F_1}\right)^2=c_{F_1F_1}\E X_{F_1^2}\big(1+O(n^{-1})\big)\\
X_{F_1}^2 = c_{F_1F_1}X_{F_1^2}\big(1+O_p(\epsilon_{F_1})\big).
\end{dcases}\end{multline*}
Therefore, we have
\begin{equation}\label{var-of-X}
\var\left(X_{F_1}\right)
= \E X_{F_1}^2-\left(\E X_{F_1}\right)^2\\
= O\left(\epsilon_{F_1}\big(\E X_{F_1}\big)^2\right).
\end{equation}
Then, since by~\cref{nb-copy-order} $\E X_{F_1}>0$, we can use Chebyshev's inequality to obtain that for any $\epsilon>0$,
\[
\P\left[\left|\frac{X_{F_1}}{\E X_{F_1}}-1\right|\geq\epsilon\right] \leq \epsilon^{-2}\frac{\var\left(X_{F_1}\right)}{\left(\E X_{F_1}\right)^2} = O\left(\E\epsilon_{F_1}\right) = O\big(1/n\rho^{m(F_1)}\big) = o(1).
\]
Therefore ${X_{F_1}}/{\E X_{F_1}}\to_p1$, which is the sought-after result.
\end{proof}

\cref{hom-prop} fully describes the distributional properties of rooted counts in the limit of large inhomogeneous random graphs: Rooted counts concentrate to their means, and their interactions concentrate to the corresponding gluing product. Further, these concentrations take place uniformly in $x$, and at least at rate $n\rho^{m(F)}$. This parallels known results for the total, unrooted count in Erd\H{o}s-R\'enyi random graphs~\citep{rucinski1988small}.

%%%%%%%%%%%%%%%%%%%%%%%%%%%%%%%%%%%%%
\section{Proof of~\cref{motif-density-limit}}%%%%%%%%%%%%%%%%%
\label{proof-motif-density-limit-SI}%%%%%%%%%%%%%%%%%%%%
%%%%%%%%%%%%%%%%%%%%%%%%%%%%%%%%%%%%%

Here we address the multivariate case, resuming from the end of the proof of~\cref{motif-density-limit} in~\cref{proof-motif-density-limit}. We consider $F = (F_1,\dots,F_k)$ a $k$-tuple of rooted graphs and still call $s_i(F,G)$ the corresponding $k$-variate vector of densities. From~\eqref{thm1-mom-conv} we know that each entry of $s_i(F,G)$ is asymptotically Normal. Remains to show that the joint distribution is also Normal. To do so, we will show convergence in moments, since it is here also sufficient to conclude~\cite{kleiber2013moment}.

Fix $r = (r_1,\dots,r_k)\in\mathbb{N}^k$. By the same arguments as above, and crucially by~\cref{cross-mom}, we have with $\Sigma$ the correlation matrix of $s_i(F,G)$ (which exists because $s_i(F,G)$ is bounded with high probability) and $[rk]$ the multiset where each $t\in[k]$ is repeated $r_t$ times, 
\[
\begin{cases}
\text{if $\textstyle{\sum_t}r_t$ is even,}&
\E \prod_{t=1}^k \left[\frac{s_i(F,G)- \E s_i(F,G)}{\sqrt{\var(s_i(F,G))}}\right]^{r_t}\to
	\sum_{s\in[rk]^{(2)}}
	\prod_{\{p,q\}\in s}\Sigma_{pq},\\[.3cm]
\text{otherwise,}&
\E \prod_{t=1}^k \left[\frac{s_i(F,G)- \E s_i(F,G)}{\sqrt{\var(s_i(F,G))}}\right]^{r_t}\to0.
\end{cases} 
\]
Thus, $s_i(F,G)$ converges in moments to $\mathrm{Normal}(0,\Sigma)$, which completes the proof.

%%%%%%%%%%%%%%%%%%%%%%%%%%%%%%%%%%%%%
\section{Application of~\cref{motif-density-clt}}%%%%%%%%%%%%%%
\label{uniform-motif-density-clt}%%%%%%%%%%%%%%%%%%%%%%
%%%%%%%%%%%%%%%%%%%%%%%%%%%%%%%%%%%%%
In \cref{ml-estimation} we claim that:

{\emph{\cref{motif-density-clt} shows that when using $(y_i, s_i(F,G))_i$ to fit a probabilistic model, statistical inference proceeds as if one were using the i.i.d. sample $(y_i, s_{x_i}(F,\kappa))_i$. To see this, fix a smooth log-likelihood function $l_\theta(\cdot,\cdot)$, and set
\[
\hat\theta_G=\arg\!\max_\theta\sum_il_\theta(Y_i, s_i(F,G)),\quad \hat\theta_\kappa=\arg\!\max_\theta\sum_il_\theta(Y_i, s_{x_i}(F,\kappa)).
\]
Then, \cref{motif-density-clt} shows that $\hat\theta_G$ and $\hat\theta_\kappa$ converge at the same rate and toward the same value. Further, the asymptotic variance of $\hat\theta_G$ may be computed in the standard way, using the Fisher information matrix.}}

These claim hold if \cref{motif-density-clt} holds uniformly in $\theta$, or at least in a neighborhood of some $\theta_0$, where $\theta_0$ is the true parameter. As is transparent from the proof of \cref{motif-density-clt}, the bound depend only on $n$ and $K$. Therefore, for this convergence to occur we need there to be a finite constant $K$ such that for all $\theta$, the bivariate maps $l_\theta(\cdot,\cdot)$ and $\partial l_\theta/\partial\theta\big(\cdot,\cdot)$ are $K$-Lipschitz in the neighborhood of $\theta_0$.

%%%%%%%%%%%%%%%%%%%%%%%%%%%%%%%%%%%%%
\section{Extensions to~\cref{motif-density-clt}}%%%%%%%%%%%%%%
\label{thm_extensions}
%%%%%%%%%%%%%%%%%%%%%%%%%%%%%%%%%%%%%
Here we present a natural extension to~\cref{motif-density-clt}, showing uniform convergence to the normal distribution for the sum of subgraph counts. We will only present the result and proof for the univariate case, and the counts themselves, but an extension to the same setting as~\cref{motif-density-clt} would follow using the same proof technique.

This following result is surprising in that one would have expected a standard Berry-Esseen theorem to hold, with concentration at the speed of $\sqrt{n}$, while we observe another, larger, error term in $\epsilon_n$. The cause of this additional term is the dependence between rooted counts. The reasons why this terms is so large is because the intuition that higher order dependence terms---terms of the form $\E X^pY^q$ for two variables $X$ and $Y$ and $p,q>0$---are small proves to be, for rooted counts, wrong. Indeed, for rooted counts, these higher order cross moments remain purely driven by $\epsilon_n$, as they remain for all order explained by the overlap between two copies.

\begin{Proposition}[Berry-Esseen type result for sum of rooted densities]%%%%%
\label{berry-eseen}
With $F$ a rooted graph, $G\sim G(\rho,\kappa)$ such that $\E s_x(F,\kappa)>0$ and $S_n=\sum_{i\in[n]}(X_F(G_n, i)-\E X_F)$. Then, if $\epsilon_n = \max\big\{1/n^{{}_{|H|-1}}\rho^{{}_{\,e(H)}}: H\subset F\big\}=o(1)$, we have that:
\[
\sup_t\left|F_{n^{-1/2}S_n}(t) - F_{Normal(0,\sigma)}(t)\right| = O\big(\sqrt{\epsilon_n}\,\big).
\]
\end{Proposition}%%%%%%
\begin{proof}%%%%%
As in the proof of~\cref{motif-density-clt}, we proceed by splitting $S_n$ into two components. However, in this case, it is more convenient to do so implicitly through factorising the characteristic function. Here we note that we shall rely on elements of proofs of earlier results in this document, and therefore will not provide as much details as for the earlier proofs.

Let us first write the characteristic function of $S_n/\sqrt{n}$, which we call $\phi_n$:
\[
\phi_n(t)=\E e^{iS_nt/\sqrt{n}}=\sum_{k\geq0}\frac{(it/\sqrt{n})^k}{k!}\E S_n^k.
\]
Quite naturally following our earlier proofs, we focus on the moments, and write:
\[
\E S_n^k=\sum_{0\leq k_1\leq\ldots\leq k_n\leq k\,;\,\sum_i k_i=k}
	\E\prod_{i\in[n]}\big(X_F(G_n, i) - \E X_F\big)^{k_i}.
\]
We now focus on how the dependence between the $X_F(G_n, i)$ inflate these terms, and at what speed this increase in dependence shrinks with $n$ (for this reason we shall not keep dependence in $k$ in the asymptotic analysis to follow.) There we observe that by the proof of~\cref{cross-mom}, $\big(X_F(G_n, i) - \E X_F\big)^{k_i}$ can be though of as the count of $k_i$ overlapping, copies of $F$ rooted at $i$. Then, the probability that any of the $k_i$ copies of $F$ rooted at $i$ overlaps with any of the $k_j$ copies of $F$ rooted at any $j\neq i$ is of order $k_i(\sum_{j\neq i} k_j)\epsilon_n/n$ (using the same logic as in the paragraph after~\eqref{CLT-proof-B1} in the proof of ~\cref{motif-density-clt}; note that $k_i(\sum_{j\neq i} k_j)\epsilon_n\to_n 0$, so we need not control for the case where this probability is larger than $1$.) As any additional overlap would lead to further negligible terms, by summing across $i$, we obtain that if $\forall i, k_i\neq 1$
\[
\E\prod_{i\in[n]}\big(X_F(G_n, i) - \E X_F\big)^{k_i}=\big(1+O(\epsilon_n/n)\big)
	\prod_{i\in[n]}\E \big(X_F(G_n, i) - \E X_F\big)^{k_i},
\]
else, since by~\cref{hom-prop}, $X_F(G_n, i) = \big(1+O_p(\epsilon_n/n)\big) \E X_F$, we have that
\[
(\epsilon_n/n)^{-1/2}\big(X_F(G_n, i) - \E X_F\big) = O(1)
\]
and therefore
\[
\E\prod_{i\in[n]}\big(X_F(G_n, i) - \E X_F\big)^{k_i} = O\big((\epsilon_n/n)^{k/2}\big).
\]
It follows that with~$S_n'$ distributed at the sum of~$n$ independent and centred copies of~$X_F$, and since there is order~$n^k$ terms where at least one~$k_i$ equals~$1$, we have:
\[
\phi_n(t)=\sum_{k>}\frac{(it/\sqrt{n})^k}{k!}\left(\big(1+O(\epsilon_n/n)\big)\E S_n'^k + O\big((n\epsilon_n)^{k/2}\big)\right).
\]
We therefore recover three terms, one which is the standard one of a sum of $i.i.d.$ random variables, one bounded by $\epsilon_n/n$, and one bounded by $\sqrt{\epsilon_n}$. Then, expressing the distribution function as the Fourier transform of the characteristic function, addressing the first term using the Berry-Esseen theorem, we recover:
\[
\sup_t\left|F_{n^{-1/2}S_n}(t) - F_{Normal(0,\sigma)}(t)\right| = O\left(n^{-1/2} + \sqrt{\epsilon_n}\right).
\]
We obtain the result notting that for any $\rho$, $n^{-1/2}=O(\epsilon_n)$.
\end{proof}

\begin{Remark}[Donsker extension to~\cref{motif-density-clt}]
\Cref{berry-eseen} implies that in large sparse kernel based random graphs the empirical process of ours statistic is not tight, so that a Donsker type extensions (as in convergence of the interpolated partial sums to a Brownian bridge) does not hold. For dense graphs, as in $\rho\to c\in (0,1)$, then $\epsilon_n$ is of order $n^{-1}$, and the sequence is tight, so that a Donsker extensions would hold. This is consistent with the intuition that the dependence between rooted counts increases with sparsity. However, it is interesting to observe that any degree of sparsity breaks this convergence. 
\end{Remark}

\begin{Remark}[Relation to~\cite{barbour2019}]
\Citet{barbour2019} presents a result related to~\cref{motif-density-clt}. In~\cite{barbour2019} averages across vertices of functions of vertex's neighborhood are found to be normal in a special case of our null (degree based models, which corresponds to rank one kernels) and in the case of a very sparse graph ($n\rho\to c>0).$) There we observe that in such a sparse setting, all but a finite number of neighborhoods will be trees, and that for $F$ a rooted tree,~\cref{motif-density-clt} holds only for sequences $\rho$ such that $n\rho\to\infty$.Thus,~\cite{barbour2019} shows that the behavior~\cref{motif-density-clt} persists even in sparser regimes. From~\cite[Theorem~1]{BickelLevina2012} we knew the global count to be Normal in the very sparse regime of $n\rho\to c>0$ (with an inflated variance compared to the denser cases.) 
\end{Remark}

%%%%%%%%%%%%%%%%%%%%%%%%%%%%%%%%%%%%%
%%%%%%%%%%%%%%%%%%%%%%%%%%%%%%%%%%%%%

\section*{Acknowledgements}
I gratefully acknowledge that the authors of~\cite{olhede2013network} kindly shared the dataset used in this work, and provided comments on early versions of this document. Discussions with members of UCL Statistical Science Department and C. J. Priebe helped improve this document. I thank anonymous referees for their constructive comments, the Isaac Newton Institute for Mathematical Sciences, Cambridge, for support and hospitality during the program Theoretical Foundations for Statistical Network Analysis (EPSRC grant no. EP/K032208/1) where part of this work was undertaken. This work was supported in part by RELx, through a partnership with the UCL Big Data Institute, and by the NCSML PDRA Challenge Best Article prize. Any remaining mistakes are mine alone.

%%%%%%%%%%%%%%%%%%%%%%%%%%%%%%%%%%%%%
\bibliographystyle{plainnat}%%%%%%%%%%%%%%%%%%%
\bibliography{LNA_bib}%%%%%%%%%%%%%%%%%%%%%%%%%
%%%%%%%%%%%%%%%%%%%%%%%%%%%%%%%%%%%%%
\end{document}